\expandafter\def\csname ver@fixltx2e.sty\endcsname{}
\documentclass[journal,10pt]{IEEEtran}

\let\proof\relax 
\let\endproof\relax
\usepackage{amsthm}
\usepackage{amsfonts}
\usepackage{amssymb,bm}
\usepackage{mathrsfs}
\usepackage{mathtools}
\usepackage{amsmath}

\usepackage{cite}
\usepackage{balance}

\let\labelindent\relax
\usepackage[inline]{enumitem}

\usepackage{verbatim}

\usepackage{dblfloatfix}
\usepackage{url}
\usepackage{graphicx}
\usepackage{textcomp}

\usepackage{accents}

\newlength{\dhatheight}
\newcommand{\doublehat}[1]{%
    \settoheight{\dhatheight}{\ensuremath{\hat{#1}}}%
    \addtolength{\dhatheight}{-0.25ex}%
    \hat{\vphantom{\rule{1pt}{\dhatheight}}%
    \smash{\hat{#1}}}}

\usepackage{array}
\newcolumntype{P}[1]{>{\centering\arraybackslash}p{#1}}

\usepackage{xcolor}

\newtheorem{theorem}{Theorem}

\newtheorem{lemma}{Lemma}

\newtheorem{corollary}[theorem]{Corollary}
\newtheorem{claim}{Claim}

\newtheorem{remark}{Remark}
\newenvironment{claimproof}{\proof}{\endproof}

\def\BibTeX{{\rm B\kern-.05em{\sc i\kern-.025em b}\kern-.08em
    T\kern-.1667em\lower.7ex\hbox{E}\kern-.125emX}}

\begin{document}

\title{Inverse Particle Filter}

\author{Himali Singh, Arpan Chattopadhyay$^\ast$ and Kumar Vijay Mishra$^\ast$\vspace{-24pt}
\thanks{$^\ast$A. C. and K. V. M. have made equal contributions.}
\thanks{H. S. and A. C. are with the Electrical Engineering Department, Indian Institute of Technology (IIT) Delhi, India. A. C. is also associated with the Bharti School of Telecommunication Technology and Management, IIT Delhi. Email: \{eez208426, arpanc\}@ee.iitd.ac.in.} 
\thanks{K. V. M. is with the United States DEVCOM Army Research Laboratory, Adelphi, MD 20783 USA. E-mail: kvm@ieee.org.}
\thanks{A. C. acknowledges support via grant no. GP/2021/ISSC/022 from I-Hub Foundation for Cobotics and grant no. CRG/2022/003707 from Science and Engineering Research Board (SERB), India. H. S. acknowledges support via Prime Minister Research Fellowship. K. V. M. acknowledges support from the National Academies of Sciences, Engineering, and Medicine via Army Research Laboratory Harry Diamond Distinguished Fellowship.}
}

\maketitle

\begin{abstract}
In cognitive systems, recent emphasis has been placed on studying the cognitive processes of the subject whose behavior was the primary focus of the system's cognitive response. This approach, known as \textit{inverse cognition}, arises in counter-adversarial applications and has motivated the development of inverse Bayesian filters. In this context, a cognitive adversary, such as a radar, uses a forward Bayesian filter to track its target of interest. An inverse filter is then employed to infer the adversary's estimate of the target's or defender's state. Previous studies have addressed this inverse filtering problem by introducing methods like the inverse Kalman filter (KF), inverse extended KF, and inverse unscented KF. However, these filters typically assume additive Gaussian noise models and/or rely on local approximations of non-linear dynamics at the state estimates, limiting their practical application. In contrast, this paper adopts a global filtering approach and presents the development of an inverse particle filter (I-PF). The particle filter framework employs Monte Carlo methods to approximate arbitrary posterior distributions. Moreover, under mild system-level conditions, the proposed I-PF demonstrates convergence to the optimal inverse filter. Additionally, we propose the differentiable I-PF to address scenarios where system information is unknown to the defender. Using the recursive Cram\'{e}r-Rao lower bound and non-credibility index, our numerical experiments for different systems demonstrate the estimation performance and time complexity of the proposed filter.
\end{abstract}


\vspace{-8pt}
\section{Introduction}
\label{sec:introduction}
Several applications in engineering, such as communications, sensing, and robotics, frequently employ \textit{cognitive} agents that perceive their surroundings and adapt their actions based on the information learned to attain optimum efficiency. A cognitive surveillance radar \cite{mishra2023next,mishra2020toward}, for example, modifies its transmit waveform and receive processing to enhance target detection \cite{mishra2017performance} and tracking \cite{bell2015cognitive,sharaga2015optimal}. In this context, \textit{inverse cognition} has recently been introduced as a method for a `defender' agent to identify its adversarial `attacker' agent's cognitive behavior and infer the information learned about the defender \cite{krishnamurthy2019how,krishnamurthy2020identifying}. This facilitates the development of counter-adversarial systems to assist or desist the adversary. For instance, an intelligent target can observe its adversarial radar's waveform adaptations and develop smart interference that forces the latter to change its course of actions \cite{krishnamurthy2021adversarial,kang2023smart}. Interactive learning, fault diagnosis, and cyber-physical security are further examples of counter-adversarial applications \cite{krishnamurthy2019how,mattila2020hmm}. A similar formulation can also be found in inverse reinforcement learning (IRL) \cite{ng2000algorithms} wherein the associated reward function is passively learned based on the behavior of an expert. In contrast, the inverse cognition agent actively explores its adversarial agent and can thus be regarded as a generalization of IRL.

Counter-adversarial applications involve \textit{inference} by both the defender and attacker, wherein they estimate the posterior distributions of an underlying state that cannot be directly observed but is inferred through an observation process conditioned on that state. Posterior distributions provide not only \textit{point estimates}, i.e., single-number estimates such as the mean of the posterior distributions, but also quantify their uncertainty. In sequential state estimation, states are inferred from a sequence of observations, with posteriors updated recursively. When state-space models are available, Bayesian filtering is a probabilistic technique for sequential state estimation that has been extensively applied to visual tracking\cite{zhang2017multi}, localization in robotics \cite{ullah2019localization}, and other signal processing problems \cite{yousefi2019efficient,liu2019audio}.
 
In inverse cognition, the attacker employs a (forward) Bayesian filter to infer the kinematic state of the defender, which the former then uses to cognitively adapt its actions. In order to predict the attacker's future actions, a defender assesses the attacker's inference using an \textit{inverse Bayesian filter} \cite{krishnamurthy2019how} that estimates the posterior distribution of the forward filter given noisy measurements of the attacker's actions. In this context, Kalman filter (KF) is a well-known Bayesian filter for linear Gaussian systems, providing optimal minimum mean-squared error (MMSE) estimates. However, posterior computation becomes intractable for general non-linear and non-Gaussian systems. To address this, approximate approaches like the extended KF (EKF) \cite{anderson2012optimal} and unscented KF (UKF) \cite{julier2004unscented} have been proposed. EKF and UKF assume Gaussian state posteriors and are only applicable to systems with Gaussian noises \cite{li2017approximate}. While they provide closed-form analytic solutions, their appropriateness varies by application, often failing with highly non-linear models or multi-modal posteriors \cite{kotecha2003gaussian,cappe2007overview}. Techniques like Gaussian mixture (GM) \cite{alspach1972nonlinear} or grid-based \cite{kramer1988recursive} filters, proposed to mitigate these limitations, become computationally expensive in high-dimensional systems. The curse of dimensionality also affects deterministic numerical integration methods like cubature KF (CKF) \cite{arasaratnam2009cubature} and quadrature KF (QKF) \cite{ito2000gaussian,arasaratnam2007qkf}, with the approximation error's convergence rate decreasing as state dimension increases \cite{crisan2002survey}.

In practice, counter-adversarial applications are often non-linear and non-Gaussian, limiting the applicability of Gaussian approximations in EKF/UKF. Non-linear filtering methods that do not rely on this assumption use sequential importance sampling (SIS) and resampling \cite{gordon1993novel}, leading to particle filtering (PF)\cite{cappe2007overview}. PF methods employ random sampling to achieve asymptotically exact integral computation, though with higher computational complexity\cite{ristic2003beyond,arulampalam2002tutorial}. Whereas EKF and UKF locally approximate posterior distributions at state estimates \cite{li2017approximate}, PFs use a global approach. They start with a set of sample points representing the initial state distribution and propagate them through the actual nonlinear dynamics, with the ensemble of these samples providing an approximate posterior. Consequently, PFs are suitable for non-Gaussian dynamics. Additionally, quasi-Monte Carlo (MC) methods \cite{guo2006quasi,avron2016quasi} are deterministic alternatives to MC methods, using regularly distributed points rather than random ones to approximate the posterior. In this paper, we develop inverse filters using the PF framework for estimating the attacker's inference in highly non-linear, non-Gaussian counter-adversarial systems.

\vspace{-11pt}
\subsection{Prior Art}
Our work is closely connected to a rich tradition of research in the development of PFs, resulting in a vast literature. However,  as detailed in the sequel, almost all previous studies have concentrated on forward filters. Following \cite{gordon1993novel}'s bootstrap PF formulation, several variants of PFs have been developed for enhanced performance. Auxiliary PFs \cite{pitt1999filtering} direct particles to higher-density regions of the posterior distribution, whereas unscented PFs \cite{van2000unscented_pf} include the current observation into the proposal distribution using UKF. In \cite{murphy2001rao}, Rao-Blackwellised PF is proposed for systems wherein the state can be partitioned such that the posterior distribution of one part is tractably conditioned on the other. The PF framework has also been used to implement Bernoulli filters \cite{ristic2013tutorial} (for randomly switching systems), possibility PFs \cite{ristic2019robust} (for mismatched models), and probability hypothesis density (PHD) filters \cite{mahler2003multitarget} (for high-dimensional multi-object Bayesian inference). 

In the context of inverse stochastic filtering, \cite{mattila2020hmm} examined finite state-space models and proposed an \textit{inverse hidden Markov model} to estimate the adversary's observations and observation likelihood. The \textit{inverse KF} (I-KF) in \cite{krishnamurthy2019how} estimates the defender's state based on a forward KF's estimation. For non-linear counter-adversarial systems, we recently developed \textit{inverse extended KF} (I-EKF) and \textit{inverse unscented KF} (I-UKF) in \cite{singh2022inverse,singh2022inverse_part1} and  \cite{singh2023counter,singh2023inverse_ukf}, respectively. In the case of I-EKF, the adversary employs a forward EKF, which utilizes Taylor series expansion of the non-linear dynamics. As a result, EKF necessitates Jacobian computation, is susceptible to initialization/modeling errors, and performs poorly when significant non-linearities are present \cite{li2017approximate}. We addressed some of these shortcomings in the context of inverse cognition using advanced variants of I-EKF \cite{singh2022inverse_part2}. I-UKF, on the other hand, is an alternative derivative-free technique for efficiently dealing with nonlinear systems. Based on the unscented transform, UKF \cite{julier2004unscented} uses a weighted sum of function evaluations at a finite number of deterministic sigma points and approximates the posterior distribution of a random variable under non-linear transformation. CKF \cite{arasaratnam2009cubature} and QKF \cite{ito2000gaussian,arasaratnam2007qkf} are further examples of derivative-free nonlinear filters that use efficient numerical integration techniques to compute the Bayesian recursive integrals. These formulations for inverse CKF and QKF were proposed and studied recently in \cite{singh2023inverse_ckf_qkf}. However, the inverses of PFs have remained unexamined in prior works.

\vspace{-11pt}
\subsection{Our Contributions}\label{subsec:contribution}
Our main contributions in this paper are as follows:\\
\textbf{1) Inverse PF.} Gaussian inverse filters such as I-EKF \cite{singh2022inverse_part1} and I-UKF \cite{singh2023inverse_ukf} are not applicable to general non-Gaussian systems. To address this, we develop inverse PF (I-PF). At the $k$-th time instant, our I-PF considers the joint conditional distribution of the attacker's current state estimate and observation given the defender's knowledge of its own true states and observations of the attacker's actions up to the current instant. Initially, we assume perfect system model information, including a general but known forward filter at the defender's end. Our I-PF seeks to empirically approximate the optimal inverse filter's (joint) posterior and samples the particles from the optimal importance sampling density. This is in contrast to the typical PF, where the optimal density is often unavailable. The known forward filter assumption allows sampling from the optimal density in I-PF; see also Remark~\ref{remark:IPF optimal density}.\\
\textbf{2) Convergence of I-PF.} In PFs, the particles interact and are not statistically independent, rendering classical convergence results for Monte Carlo methods, which rely on central limit theorems under i.i.d. assumptions, inapplicable. Despite this, it is essential to study PFs' convergence to the true posterior as they approximate optimal filters. In this work, we examine the convergence of our proposed inverse particle filter (I-PF) in the $L^{4}$-sense. Specifically, we demonstrate that our I-PF's estimates converge to the optimal inverse filter's estimates for bounded observation densities, given that the estimated function grows at a slower rate than the defender's observation density. Moreover, convergence in the $L^{4}$-sense also implies almost sure convergence of our I-PF.\\
\textbf{3) Differentiable I-PF.} The applications of the aforementioned inverse filters, including prior works \cite{krishnamurthy2019how,singh2022inverse_part1,singh2023inverse_ukf} are limited to cases when perfect system information is available. However, in practical counter-adversarial systems, the forward filter and the strategy adopted by the attacker to adapt its actions may not be known to the defender. Similarly, the attacker may not have complete information about the defender's state evolution. We address this case of unknown dynamics by proposing a differentiable I-PF that leverages learning networks to learn both the model and I-PF's proposal distribution.\\
\textbf{4) Recursive lower error bounds.} We consider a widely used one-dimensional non-linear system\cite{arulampalam2002tutorial,hu2008basic}, a bearing-only tracking system\cite{bar2004estimation,lin2002comparison} and a non-Gaussian system with non-stationary observations\cite{van2000unscented_pf} to demonstrate the estimation performance and time complexity of our I-PF in comparison to the I-EKF\cite{singh2022inverse_part1} and I-UKF\cite{singh2023inverse_ukf}. We evaluate the estimation performance using the recursive Cram\'{e}r-Rao lower bound (RCRLB)\cite{tichavsky1998posterior} and the non-credibility index (NCI)\cite{li2001practical} as key performance metrics. Our numerical experiments indicate that the proposed I-PF provides better estimates, especially when assuming a forward filter different from the attacker's actual forward filter, and is more credible than the I-EKF. 

The rest of the paper is organized as follows. The next section describes the system model and develops the optimal inverse filter recursions. In Section~\ref{sec:IPF}, we develop I-PF while its convergence results are provided in Section~\ref{sec:ipf convergence}. Section~\ref{sec:unknown} presents the differentiable I-PF for the unknown system case. We discuss numerical experiments for the proposed filter's performance in Section~\ref{sec:simulation}, before concluding in Section~\ref{sec:summary}.

Throughout the paper, we reserve boldface lowercase and uppercase letters for vectors (column vectors) and matrices, respectively, and $\lbrace a_{i}\rbrace_{i_{1}\leq i\leq i_{2}}$ denotes a set of elements indexed by an integer $i$. The notation $[\mathbf{a}]_{i}$ is used to denote the $i$-th component of vector $\mathbf{a}$. The transpose operation is $(\cdot)^{T}$, expectation operation is $\mathbb{E}[\cdot]$ and the $l_{2}$ norm of a vector is $\|\cdot\|$. The notation $\|f\|_{\infty}$ denotes the supremum norm of the real-valued function $f(\cdot)$. Also, $\mathbf{I}_{n}$ and $\mathbf{0}$ denote a `$n\times n$' identity matrix and an all-zero matrix, respectively. The function $\delta(x-x_{0})$ is the Dirac-delta function in variable $x$ centered at $x_{0}$. The Gaussian distribution is represented as $\mathcal{N}(\mathbf{x};\bm{\mu},\mathbf{Q})$ with mean $\bm{\mu}$ and covariance matrix $\mathbf{Q}$ while $\mathcal{U}[a,b]$ represents a uniform distribution over interval $[a,b]$. We use shorthand $\mathbb{P}(\mathbf{X}\in d\mathbf{x})$ to refer to the probability of random variable $\mathbf{X}\in [\mathbf{x},\mathbf{x}+d\mathbf{x}]$, where $d\mathbf{x}$ is an infinitesimal interval length.

\vspace{-8pt}
\section{Optimal Inverse Filter}
\label{sec:optimal filter}
Before considering the approximate PF framework in the subsequent section, we derive the optimal Bayesian recursions for inverse filtering. We consider a general probabilistic framework for the attacker-defender dynamics in Section~\ref{subsec:system model}. The optimal inverse filter recursions are provided in Section~\ref{subsec:optimal recursions}.

\vspace{-8pt}
\subsection{System model}
\label{subsec:system model}
Consider the `$n_{x}$'-dimensional stochastic process $\mathbf{X}=\{\mathbf{X}_{k}\}_{k\geq 0}$ as the defender's state evolution process. Following the inverse cognition framework of \cite{mattila2020hmm,krishnamurthy2019how,singh2022inverse_part1}, the defender perfectly knows its state $\mathbf{x}_{k}\in\mathbb{R}^{n_{x}\times 1}$ for all $k\geq 0$. For instance, an intelligent target is aware of its own position and velocity at all times. The process $\mathbf{X}$ is a Markov process with initial state $\mathbf{X}_{0}\sim\pi^{x}_{0}(d\mathbf{x}_{0})$ and evolves as
\par\noindent\small
\begin{align}
    \mathbb{P}(\mathbf{X}_{k+1}\in d\mathbf{x}_{k+1}|\mathbf{X}_{k}=\mathbf{x}_{k})=\mathcal{K}(\mathbf{x}_{k+1}|\mathbf{x}_{k})d\mathbf{x}_{k+1},\label{eqn:state x}
\end{align}
\normalsize
where $\mathcal{K}(\cdot)$ denotes the transition kernel density (with respect to a Lebesgue measure). The attacker observes the defender's state as a `$n_{y}$'-dimensional observation process $\mathbf{Y}=\{\mathbf{Y}_{k}\}_{k\geq 1}$. The observations $\mathbf{Y}$ are conditionally independent given $\mathbf{X}$ with
\par\noindent\small
\begin{align}
    \mathbb{P}(\mathbf{Y}_{k}\in d\mathbf{y}_{k}|\mathbf{X}_{k}=\mathbf{x}_{k})=\rho(\mathbf{y}_{k}|\mathbf{x}_{k})d\mathbf{y}_{k},\label{eqn:observation y}
\end{align}
\normalsize
where $\rho(\cdot)$ is the attacker's conditional observation density and $k$-th observation $\mathbf{y}_{k}\in\mathbb{R}^{n_{y}\times 1}$.

The attacker computes an estimate $\hat{\mathbf{x}}_{k}$ of the defender's state $\mathbf{x}_{k}$ given the available observations $\{\mathbf{y}_{j}\}_{1\leq j\leq k}$ using the forward filter. Consider $\hat{\mathbf{X}}=\{\hat{\mathbf{X}}_{k}\}_{k\geq 0}$ as the attacker's state estimation process. The forward filter recursively computes the current estimate $\hat{\mathbf{x}}_{k}$ from the previous estimate $\hat{\mathbf{x}}_{k-1}$ and current observation $\mathbf{y}_{k}$ in a deterministic manner as
\par\noindent\small
\begin{align}
   \hat{\mathbf{x}}_{k}=T(\hat{\mathbf{x}}_{k-1},\mathbf{y}_{k}).\label{eqn:filter T}
\end{align}
\normalsize
For instance, $T(\cdot)$ represents the standard EKF/UKF recursive update if the attacker employs a forward EKF/UKF to compute state estimate $\hat{\mathbf{x}}_{k}$. Note that $T(\cdot)$ can be a time-dependent function for many forward filters. In the case of EKF/UKF, the mapping $T(\cdot)$ at the $k$-th time instant depends on the covariance matrix estimate computed at the previous $(k-1)$-th time instant. For the sake of brevity, we simply denote the forward filter recursion as in \eqref{eqn:filter T} but implement the appropriate function for the given time instant. The attacker then uses the estimate $\hat{\mathbf{x}}_{k}$ to administer an action which the defender observes as a `$n_{a}$'-dimensional noisy observation process $\mathbf{A}=\{\mathbf{A}_{k}\}_{k\geq 1}$. Given $\hat{\mathbf{X}}$, the observations $\mathbf{A}$ are conditionally independent and
\par\noindent\small
\begin{align}
    \mathbb{P}(\mathbf{A}_{k}\in d\mathbf{a}_{k}|\hat{\mathbf{X}}_{k}=\hat{\mathbf{x}}_{k})=\beta(\mathbf{a}_{k}|\hat{\mathbf{x}}_{k})d\mathbf{a}_{k},\label{eqn:observation a}
\end{align}
\normalsize
where $\beta(\cdot)$ is the defender's observation density and the $k$-th observation $\mathbf{a}_{k}\in\mathbb{R}^{n_{a}\times 1}$. Finally, the defender uses $\{\mathbf{a}_{j},\mathbf{x}_{j}\}_{1\leq j\leq k}$ to compute the estimate $\doublehat{\mathbf{x}}_{k}$ of $\hat{\mathbf{x}}_{k}$ in the inverse filter. Fig.~\ref{fig:system model} graphically illustrates the system dynamics.
%---------------------------------------------------------------------
\begin{figure}
  \centering
  \includegraphics[width = 1.0\columnwidth]{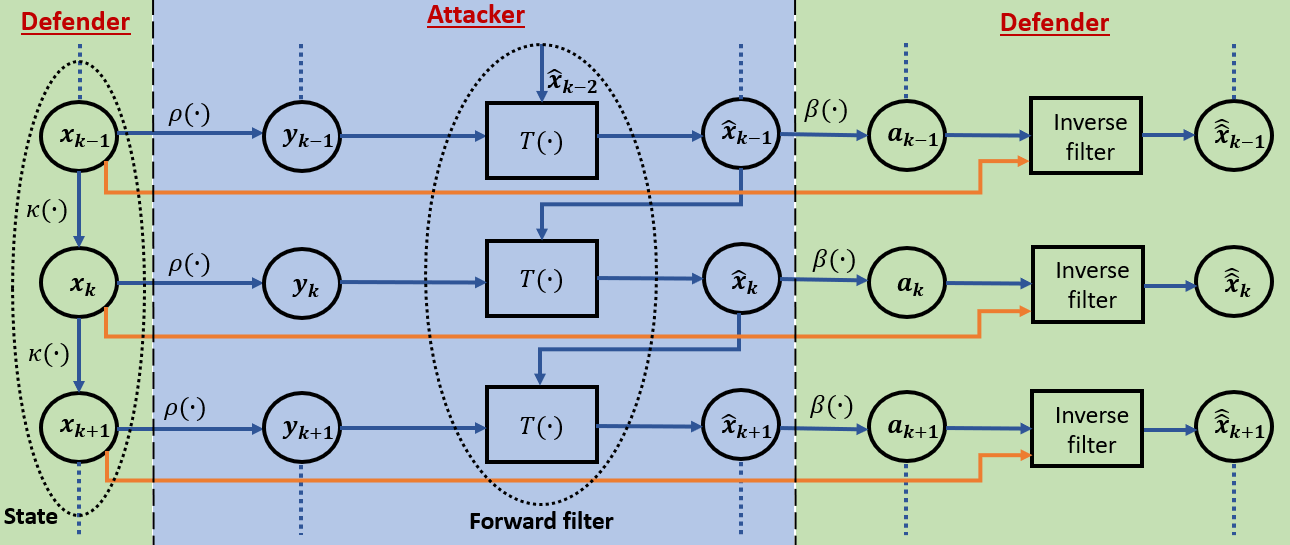}
  \caption{Graphical representation of the forward and inverse filters' recursions.}
 \label{fig:system model}
\end{figure}
%-----------------------------------------------------

In Section~\ref{sec:IPF}, we assume that both defender and attacker have perfect knowledge of the system model, i.e., the densities $\mathcal{K}(\cdot)$, $\rho(\cdot)$ and $\beta(\cdot)$. Additionally, the defender assumes a known forward filter $T(\cdot)$ employed by the attacker. We address the unknown system dynamics case in Section~\ref{sec:unknown} wherein we estimate both the state and the model parameters. Furthermore, our numerical experiments in Section~\ref{sec:simulation} show that the developed I-PF provides reasonably accurate estimates even when assuming a simple forward EKF, regardless of the attacker's actual forward filter.

\textit{Example:} Consider the special case of additive system noises such that the state evolution and observations are modeled as
\par\noindent\small
\begin{align}
    &\mathbf{x}_{k+1}=f(\mathbf{x}_{k})+\mathbf{w}_{k},\label{eqn:Gaussian state transition}\\
    &\mathbf{y}_{k}=h(\mathbf{x}_{k})+\mathbf{v}_{k},\label{eqn:Gaussian observation y}\\
    &\mathbf{a}_{k}=g(\hat{\mathbf{x}}_{k})+\bm{\epsilon}_{k}.\label{eqn:Gaussian observation a}
\end{align}
\normalsize
Here, $f(\cdot)$, $h(\cdot)$ and $g(\cdot)$ represent general non-linear functions while $\{\mathbf{w}_{k}\}$, $\{\mathbf{v}_{k}\}$ and $\{\bm{\epsilon}_{k}\}$ are mutually independent noise terms. If the probability density functions of $\mathbf{w}_{k}$, $\mathbf{v}_{k}$ and $\bm{\epsilon}_{k}$ are denoted by $p_{w}(\cdot)$, $p_{v}(\cdot)$ and $p_{\epsilon}(\cdot)$, respectively, then $\mathcal{K}(\mathbf{x}_{k+1}|\mathbf{x}_{k})=p_{w}(\mathbf{x}_{k+1}-f(\mathbf{x}_{k}))$, $\rho(\mathbf{y}_{k}|\mathbf{x}_{k})=p_{v}(\mathbf{y}_{k}-h(\mathbf{x}_{k}))$, and $\beta(\mathbf{a}_{k}|\hat{\mathbf{x}}_{k})=p_{\epsilon}(\mathbf{a}_{k}-g(\hat{\mathbf{x}}_{k}))$.

\vspace{-8pt}
\subsection{Optimal Inverse filter}
\label{subsec:optimal recursions}
Consider the (joint) conditional distribution of $(\hat{\mathbf{x}}_{k},\mathbf{y}_{k})$ given the defender's knowledge of its own true states and observations of the attacker's actions at the $k$-th time instant. As in forward Bayesian filter, the optimal inverse filter computes this conditional distribution recursively using the time and measurement updates. Define the conditional distributions $\pi_{k|k-1}(d\hat{\mathbf{x}}_{k},d\mathbf{y}_{k})$ and $\pi_{k|k}(d\hat{\mathbf{x}}_{k},d\mathbf{y}_{k})$ as
\par\noindent\small
\begin{align}
&\pi_{k|k-1}(d\hat{\mathbf{x}}_{k},d\mathbf{y}_{k})\doteq p(\hat{\mathbf{x}}_{k},\mathbf{y}_{k}|\mathbf{x}_{0:k},\mathbf{a}_{1:k-1})d\hat{\mathbf{x}}_{k}d\mathbf{y}_{k},\label{eqn:pi predict}\\
&\pi_{k|k}(d\hat{\mathbf{x}}_{k},d\mathbf{y}_{k})\doteq p(\hat{\mathbf{x}}_{k},\mathbf{y}_{k}|\mathbf{x}_{0:k},\mathbf{a}_{1:k})d\hat{\mathbf{x}}_{k}d\mathbf{y}_{k}.\label{eqn:pi update}
\end{align}
\normalsize
In the time-update step, we obtain $\pi_{k|k-1}$ from $\pi_{k-1|k-1}$ considering the true states $\mathbf{x}_{0:k}$ but observations $\mathbf{a}_{1:k-1}$ excluding the current observation $\mathbf{a}_{k}$. Finally, $\pi_{k|k-1}$ is updated using the current observation $\mathbf{a}_{k}$ in the measurement update step to obtain $\pi_{k|k}$. The defender's MMSE estimate $\doublehat{\mathbf{x}}_{k}$ is then $\doublehat{\mathbf{x}}_{k}=\iint\hat{\mathbf{x}}_{k}\pi_{k|k}(d\hat{\mathbf{x}}_{k},d\mathbf{y}_{k})$.

First, we consider the optimal filter's time update. We have
\par\noindent\small
\begin{align}
&p(\hat{\mathbf{x}}_{k},\mathbf{y}_{k}|\mathbf{x}_{0:k},\mathbf{a}_{1:k-1})=\nonumber\\
&\int p(\hat{\mathbf{x}}_{k},\mathbf{y}_{k}|\mathbf{x}_{0:k},\mathbf{a}_{1:k-1},\hat{\mathbf{x}}_{k-1})p(\hat{\mathbf{x}}_{k-1}|\mathbf{x}_{0:k},\mathbf{a}_{1:k-1})d\hat{\mathbf{x}}_{k-1}.\label{eqn:time update start}
\end{align}
\normalsize
However, $p(\hat{\mathbf{x}}_{k},\mathbf{y}_{k}|\mathbf{x}_{0:k},\mathbf{a}_{1:k-1},\hat{\mathbf{x}}_{k-1})=p(\hat{\mathbf{x}}_{k}|\mathbf{y}_{k},\mathbf{x}_{0:k},\mathbf{a}_{1:k-1},\hat{\mathbf{x}}_{k-1})p(\mathbf{y}_{k}|\mathbf{x}_{0:k},\mathbf{a}_{1:k-1},\hat{\mathbf{x}}_{k-1})=p(\hat{\mathbf{x}}_{k}|\mathbf{y}_{k},\hat{\mathbf{x}}_{k-1})p(\mathbf{y}_{k}|\mathbf{x}_{k}),$ because observation $\mathbf{y}_{k}$ is conditionally independent given state $\mathbf{x}_{k}$ with $p(\mathbf{y}_{k}|\mathbf{x}_{k})$ given by \eqref{eqn:observation y}. Also, the estimate $\hat{\mathbf{x}}_{k}$ is independent of $\{\mathbf{x}_{0:k},\mathbf{a}_{1:k-1}\}$ given $\{\mathbf{y}_{k},\hat{\mathbf{x}}_{k-1}\}$. In particular, $\hat{\mathbf{x}}_{k}$ is a deterministic function of $\mathbf{y}_{k}$ and $\hat{\mathbf{x}}_{k-1}$ such that $p(\hat{\mathbf{x}}_{k}|\mathbf{y}_{k},\hat{\mathbf{x}}_{k-1})=\delta(\hat{\mathbf{x}}_{k}-T(\hat{\mathbf{x}}_{k-1},\mathbf{y}_{k}))$. Hence, \eqref{eqn:time update start} yields
\par\noindent\small
\begin{align}
&p(\hat{\mathbf{x}}_{k},\mathbf{y}_{k}|\mathbf{x}_{0:k},\mathbf{a}_{1:k-1})=\int\delta(\hat{\mathbf{x}}_{k}-T(\hat{\mathbf{x}}_{k-1},\mathbf{y}_{k}))\nonumber\\
&\;\;\;\times\rho(\mathbf{y}_{k}|\mathbf{x}_{k})p(\hat{\mathbf{x}}_{k-1}|\mathbf{x}_{0:k},\mathbf{a}_{1:k-1})d\hat{\mathbf{x}}_{k-1}.\label{eqn:time update inter}
\end{align}
\normalsize
Now, $p(\hat{\mathbf{x}}_{k-1},\mathbf{y}_{k-1}|\mathbf{x}_{0:k},\mathbf{a}_{1:k-1})=p(\hat{\mathbf{x}}_{k-1},\mathbf{y}_{k-1}|\mathbf{x}_{0:k-1},\mathbf{a}_{1:k-1})$ because $\{\hat{\mathbf{x}}_{k-1},\mathbf{y}_{k-1}\}$ do not depend on future state $\mathbf{x}_{k}$ given $\mathbf{x}_{0:k-1}$ and $\mathbf{a}_{1:k-1}$. Hence, using this and \eqref{eqn:pi update}, we have the marginal distribution $\mathbb{P}(\hat{\mathbf{x}}_{k-1}\in d\hat{\mathbf{x}}_{k-1}|\mathbf{x}_{0:k},\mathbf{a}_{1:k-1})=\int\pi_{k-1|k-1}(d\hat{\mathbf{x}}_{k-1},d\mathbf{y}_{k-1})$. Substituting in \eqref{eqn:time update inter}, the optimal time update becomes
\par\noindent\small
\begin{align}
&p(\hat{\mathbf{x}}_{k},\mathbf{y}_{k}|\mathbf{x}_{0:k},\mathbf{a}_{1:k-1})=\iint\delta(\hat{\mathbf{x}}_{k}-T(\hat{\mathbf{x}}_{k-1},\mathbf{y}_{k}))\rho(\mathbf{y}_{k}|\mathbf{x}_{k})\nonumber\\
&\;\;\;\times\pi_{k-1|k-1}(d\hat{\mathbf{x}}_{k-1},d\mathbf{y}_{k-1}).\label{eqn:time update}
\end{align}
\normalsize

Now, consider the measurement update with current observation $\mathbf{a}_{k}$. The joint conditional distribution $p(\hat{\mathbf{x}}_{k},\mathbf{y}_{k},\mathbf{a}_{k}|\mathbf{x}_{0:k},\mathbf{a}_{1:k-1})=\beta(\mathbf{a}_{k}|\hat{\mathbf{x}}_{k})\pi_{k|k-1}(d\hat{\mathbf{x}}_{k},d\mathbf{y}_{k})$ because observation $\mathbf{a}_{k}$ is conditionally independent of everything else given $\hat{\mathbf{x}}_{k}$ with $p(\mathbf{a}_{k}|\hat{\mathbf{x}}_{k})$ given by \eqref{eqn:observation a}. Hence, using Bayes' theorem, we have
\par\noindent\small
\begin{align}
\pi_{k|k}(d\hat{\mathbf{x}}_{k},d\mathbf{y}_{k})=\frac{\beta(\mathbf{a}_{k}|\hat{\mathbf{x}}_{k})\pi_{k|k-1}(d\hat{\mathbf{x}}_{k},d\mathbf{y}_{k})}{\iint\beta(\mathbf{a}_{k}|\hat{\mathbf{x}}_{k})\pi_{k|k-1}(d\hat{\mathbf{x}}_{k},d\mathbf{y}_{k})},\label{eqn:measurement update}
\end{align}
\normalsize
which is the optimal measurement update. Table~\ref{tbl:notations} summarizes the variables and distributions defined in this section.
%--------------------------------------------------------------------
    \begin{table}
    \caption{Key variables and distributions.}
    \label{tbl:notations}
    \centering
    \begin{tabular}{p{2.2cm}p{6.0cm}}
    \hline\noalign{\smallskip}
    Notation & Description\\
    \noalign{\smallskip}
    \hline
    \noalign{\smallskip}
    $\mathbf{x}_{k}$ & Defender's state at $k$-th time instant ($\in\mathbb{R}^{n_{x}\times 1}$)\\
    $\mathbf{y}_{k}$ & Attacker's observation of $\mathbf{x}_{k}$ at $k$-th time instant ($\in\mathbb{R}^{n_{y}\times 1}$)\\
    $\mathbf{a}_{k}$ & Defender's observation of attacker's action at $k$-th time instant ($\in\mathbb{R}^{n_{a}\times 1}$)\\
    $\hat{\mathbf{x}}_{k}$ & Attacker's estimate of $\mathbf{x}_{k}$ computed via forward filter\\
    $\doublehat{\mathbf{x}}_{k}$ & Defender's estimate of $\hat{\mathbf{x}}_{k}$ computed via inverse filter\\
    $\pi^{x}_{0}(d\mathbf{x}_{0})$ & Defender's initial state distribution\\
    $\mathcal{K}(\cdot)$ & Transitional kernel density for defender's state evolution\\
    $\rho(\cdot)$ & Attacker's conditional observation density\\
    $T(\cdot,\cdot)$ & Forward filter's recursive update\\
    $\beta(\cdot)$ & Defender's conditional observation density\\
    $\pi_{k|k-1}(d\hat{\mathbf{x}}_{k},d\mathbf{y}_{k})$ & Optimal inverse filter's prediction distribution, i.e., joint conditional density of $(\hat{\mathbf{x}}_{k},\mathbf{y}_{k})$ given true states $\mathbf{x}_{0:k}$ (upto time $k$) and observations $\mathbf{a}_{1:k-1}$ (upto time $k-1$)\\
    $\pi_{k|k}(d\hat{\mathbf{x}}_{k},d\mathbf{y}_{k})$ & Optimal inverse filter's posterior distribution, i.e., joint conditional density of $(\hat{\mathbf{x}}_{k},\mathbf{y}_{k})$ given true states $\mathbf{x}_{0:k}$ (upto time $k$) and observations $\mathbf{a}_{1:k}$ (upto time $k$)\\
    \noalign{\smallskip}
    \hline\noalign{\smallskip}
    \end{tabular}
    \end{table}
 %----------------------------------------------------------------

In general, the defender estimates a function $\phi(\hat{\mathbf{x}},\mathbf{y})$ as $\mathbb{E}[\phi(\hat{\mathbf{x}},\mathbf{y})|\mathbf{x}_{0:k},\mathbf{a}_{1:k}]$ using the inverse filter's posterior distribution. For instance, considering $\phi(\hat{\mathbf{x}},\mathbf{y})=\hat{\mathbf{x}}$ yields the defender's MMSE estimate $\doublehat{\mathbf{x}}_{k}$ of $\hat{\mathbf{x}}_{k}$. For the sake of the convergence analysis in Section~\ref{sec:ipf convergence}, we now introduce some simplified notations for the integrals involved in the optimal filter recursions. Given a measure $\nu$, a function $\phi$ and a Markov transition kernel $\mathcal{K}$, we define
\par\noindent\small
\begin{align*}
\langle\nu,\phi\rangle\doteq\int\phi(x)\nu(dx),\;\;\;\mathcal{K}\phi(x)\doteq\int\phi(z)\mathcal{K}(dz|x).
\end{align*}
\normalsize
With this notation, the optimal filter recursions \eqref{eqn:time update} and \eqref{eqn:measurement update} can be expressed as
\par\noindent\small
\begin{align}
\langle\pi_{k|k-1},\phi\rangle&=\langle\pi_{k-1|k-1},\delta_{T}\rho\phi\rangle,\label{eqn:optimal filter time update}\\
\langle\pi_{k|k},\phi\rangle&=\frac{\langle\pi_{k|k-1},\beta\phi\rangle}{\langle\pi_{k|k-1},\beta\rangle},\label{eqn:optimal filter measurement update}
\end{align}
\normalsize
where $\delta_{T}$ denotes function $\delta(\hat{\mathbf{x}}_{k}-T(\hat{\mathbf{x}}_{k-1},\mathbf{y}_{k}))$. For brevity, we drop the time parameter $k$ in our notation $\delta_{T}$, but while implementing, it is taken into consideration. Note that the optimal inverse filter exists only if $\langle\pi_{k|k-1},\beta\rangle>0$.

\vspace{-8pt}
\section{Inverse PF}\label{sec:IPF}
In PF, the posterior is approximated empirically using a weighted particle set, evolving randomly in time according to the system dynamics. Such an approximation simplifies the computation of integrals to finite sums. The particles are sampled from a suitable importance density and resampled to avoid particle degeneracy \cite{ristic2003beyond,arulampalam2002tutorial}. This SIS algorithm is also known as bootstrap filtering \cite{gordon1993novel}, condensation algorithm \cite{maccormick2000probabilistic}, and interacting particle approximations \cite{del1998measure}. In the following, we develop I-PF recursions applying SIS and resampling methods for the defender-attacker dynamics, analogous to the standard PF. For further details on SIS-based filtering, we refer the readers to  \cite{arulampalam2002tutorial} and \cite{chen2003bayesian}.

\vspace{-8pt}
\subsection{I-PF formulation}\label{subsec:ipf formulation}
Consider $N$ as the total number of particles. As noted earlier in Section~\ref{subsec:optimal recursions}, in I-PF, we approximate the (joint) posterior distribution of $(\hat{\mathbf{x}}_{k},\mathbf{y}_{k})$ at $k$-th time instant as
\par\noindent\small
\begin{align*}
p(\hat{\mathbf{x}}_{k},\mathbf{y}_{k}|\mathbf{x}_{0:k},\mathbf{a}_{1:k})\approx\sum_{i=1}^{N}\omega^{i}_{k}\delta(\hat{\mathbf{x}}_{k}-\hat{\mathbf{x}}^{i}_{k},\mathbf{y}_{k}-\mathbf{y}^{i}_{k}),
\end{align*}
\normalsize
where $\{\hat{\mathbf{x}}^{i}_{k},\mathbf{y}^{i}_{k}\}_{1\leq i\leq N}$ are the current particles with associated importance weights $\{\omega^{i}_{k}\}_{1\leq i\leq N}$ and the defender knows $\mathbf{x}_{0:k}$ and $\mathbf{a}_{1:k}$. To this end, we first consider the joint conditional density $p(\hat{\mathbf{x}}_{0:k},\mathbf{y}_{1:k}|\mathbf{x}_{0:k},\mathbf{a}_{1:k})$, which we approximate using particles $\{\hat{\mathbf{x}}^{i}_{0:k},\mathbf{y}^{i}_{1:k}\}_{1\leq i\leq N}$. The particle $(\hat{\mathbf{x}}^{i}_{k},\mathbf{y}^{i}_{k})$ for approximating the marginal distribution $p(\hat{\mathbf{x}}_{k},\mathbf{y}_{k}|\mathbf{x}_{0:k},\mathbf{a}_{1:k})$ is then simply the sub-vector corresponding to $(\hat{\mathbf{x}}_{k},\mathbf{y}_{k})$ in the $i$-th particle $(\hat{\mathbf{x}}^{i}_{0:k},\mathbf{y}^{i}_{1:k})$. Furthermore, as we will show eventually, our I-PF algorithm only requires storing particles $\{\hat{\mathbf{x}}^{i}_{k}\}$, i.e., the previous particles $\{\hat{\mathbf{x}}^{i}_{0:k-1},\mathbf{y}^{i}_{1:k-1}\}$ and the current observation particles $\{\mathbf{y}^{i}_{k}\}$ are discarded. Denote $q(\cdot)$ as the chosen importance sampling density. Based on importance sampling\cite{ristic2003beyond}, the weights are computed as
\par\noindent\small
\begin{align}
\omega^{i}_{k}=\frac{p(\hat{\mathbf{x}}^{i}_{0:k},\mathbf{y}^{i}_{1:k}|\mathbf{x}_{0:k},\mathbf{a}_{1:k})}{q(\hat{\mathbf{x}}^{i}_{0:k},\mathbf{y}^{i}_{1:k}|\mathbf{x}_{0:k},\mathbf{a}_{1:k})},\label{eqn:IPF weight start}
\end{align}
\normalsize
for $i=1,2,\hdots,N$. In our inverse filtering problem, the joint density simplifies to
\par\noindent\small
\begin{align}
&p(\hat{\mathbf{x}}_{0:k},\mathbf{y}_{1:k}|\mathbf{x}_{0:k},\mathbf{a}_{1:k})\propto p(\mathbf{a}_{k}|\hat{\mathbf{x}}_{k})p(\hat{\mathbf{x}}_{k}|\hat{\mathbf{x}}_{k-1},\mathbf{y}_{k})p(\mathbf{y}_{k}|\mathbf{x}_{k})\nonumber\\
&\;\;\times p(\hat{\mathbf{x}}_{0:k-1},\mathbf{y}_{1:k-1}|\mathbf{x}_{0:k-1},\mathbf{a}_{1:k-1}).\label{eqn:ipf p density}
\end{align}
\normalsize
Furthermore, analogous to the standard PF \cite{ristic2003beyond}, we choose a sampling density $q(\cdot)$ that factorizes as
\par\noindent\small
\begin{align}
&q(\hat{\mathbf{x}}_{0:k},\mathbf{y}_{1:k}|\mathbf{x}_{0:k},\mathbf{a}_{1:k})=q(\hat{\mathbf{x}}_{k},\mathbf{y}_{k}|\hat{\mathbf{x}}_{k-1},\mathbf{y}_{k-1},\mathbf{x}_{k},\mathbf{a}_{k})\nonumber\\
&\;\;\;\times q(\hat{\mathbf{x}}_{0:k-1},\mathbf{y}_{1:k-1}|\mathbf{x}_{0:k-1},\mathbf{a}_{1:k-1}).\label{eqn:ipf q density}
\end{align}
\normalsize
We provide the detailed steps to obtain \eqref{eqn:ipf p density} and \eqref{eqn:ipf q density} in Appendix~\ref{app:sampling}. Substituting \eqref{eqn:ipf p density} and \eqref{eqn:ipf q density} in \eqref{eqn:IPF weight start}, the optimal weights simplifies to
\par\noindent\small
\begin{align}
\omega^{i}_{k}\propto\omega^{i}_{k-1}\frac{p(\mathbf{a}_{k}|\hat{\mathbf{x}}^{i}_{k})p(\hat{\mathbf{x}}^{i}_{k}|\hat{\mathbf{x}}^{i}_{k-1},\mathbf{y}^{i}_{k})p(\mathbf{y}^{i}_{k}|\mathbf{x}_{k})}{q(\hat{\mathbf{x}}^{i}_{k},\mathbf{y}^{i}_{k}|\hat{\mathbf{x}}^{i}_{k-1},\mathbf{y}^{i}_{k-1},\mathbf{x}_{k},\mathbf{a}_{k})},\label{eqn:IPF weight 2}
\end{align}
\normalsize
However, for this choice of importance density, the variance of weights can only increase over time resulting in particle degeneracy \cite{doucet2000sequential}, i.e., with time, all but one particle will have negligible weights. To this end, resampling is employed to eliminate particles with small weights and multiply the ones with large weights. The optimal $q(\cdot)$ that minimizes the variance of weights \eqref{eqn:IPF weight 2} can be obtained as \cite{doucet2000sequential}
\par\noindent\small
\begin{align}
q^{*}(\hat{\mathbf{x}}_{k},\mathbf{y}_{k}|\hat{\mathbf{x}}_{k-1},\mathbf{y}_{k-1},\mathbf{x}_{k},\mathbf{a}_{k})&=p(\hat{\mathbf{x}}_{k},\mathbf{y}_{k}|\hat{\mathbf{x}}_{k-1},\mathbf{y}_{k-1},\mathbf{x}_{k},\mathbf{a}_{k})\nonumber\\
&\approx p(\hat{\mathbf{x}}_{k}|\hat{\mathbf{x}}_{k-1},\mathbf{y}_{k})p(\mathbf{y}_{k}|\mathbf{x}_{k}),\label{eqn:optimal density}
\end{align}
\normalsize
where we have ignored the correlation between attacker's observation $\mathbf{y}_{k}$ and defender's observation $\mathbf{a}_{k}$ via state estimate $\hat{\mathbf{x}}_{k}$ and hence, the approximation. The perfect knowledge of distribution $p(\mathbf{y}_{k}|\mathbf{x}_{k})$ from \eqref{eqn:observation y} can compensate for this approximation to some extent. For sampling density $q^{*}(\cdot)$, the weights are computed from \eqref{eqn:IPF weight 2} as
\par\noindent\small
\begin{align}
\omega^{i}_{k}\propto\omega^{i}_{k-1}p(\mathbf{a}_{k}|\hat{\mathbf{x}}^{i}_{k}).\label{eqn:optimal weights}
\end{align}
\normalsize

\vspace{-8pt}
\subsection{I-PF recursions}\label{subsec:IPF recursion}
Consider the sampling density $q^{*}(\cdot)$ from \eqref{eqn:optimal density}. We have $p(\hat{\mathbf{x}}_{k}|\hat{\mathbf{x}}_{k-1},\mathbf{y}_{k})=\delta(\hat{\mathbf{x}}_{k}-T(\hat{\mathbf{x}}_{k-1},\mathbf{y}_{k}))$ from \eqref{eqn:filter T} and $p(\mathbf{y}_{k}|\mathbf{x}_{k})$ is given by \eqref{eqn:observation y}. Here, we consider resampling at each time step and $p(\mathbf{a}_{k}|\hat{\mathbf{x}}_{k})$ in \eqref{eqn:optimal weights} is given by \eqref{eqn:observation a}. Note that particles $\{\mathbf{y}^{i}_{k}\}$ are sampled from \eqref{eqn:observation y} using the true state $\mathbf{x}_{k}$ only. Further, since defender knows $\mathbf{x}_{k}$, density \eqref{eqn:optimal density} and weights \eqref{eqn:optimal weights} do not require previous particles $\{\mathbf{y}^{i}_{k-1}\}$ and hence, they need not be stored for the next recursion.

Initialize $\hat{\mathbf{x}}^{i}_{0}\sim\widetilde{\pi}^{x}_{0}(d\hat{\mathbf{x}}_{0})$ where $\widetilde{\pi}^{x}_{0}$ is the initial distribution assumed by the defender for the forward filter's initial estimate $\hat{\mathbf{x}}_{0}$. At time $(k-1)$, we have particles $\{\hat{\mathbf{x}}^{i}_{k-1}\}_{1\leq i\leq N}$ with equal weights ($1/N$) because of resampling. Finally, the I-PF recursions to compute the updated particles $\{\hat{\mathbf{x}}^{i}_{k}\}_{1\leq i\leq N}$ are as follows.\\
\textit{1) SIS:} For $i=1,2,\hdots,N$, draw i.i.d. observation particles $\overline{\mathbf{y}}^{i}_{k}\sim\rho(\mathbf{y}_{k}|\mathbf{x}_{k})$ and obtain state estimate particles $\hat{\overline{\mathbf{x}}}^{i}_{k}=T(\hat{\mathbf{x}}^{i}_{k-1},\overline{\mathbf{y}}^{i}_{k})$.\\
\textit{2) Modification:} For a given threshold $\gamma_{k}>0$, check if $\frac{1}{N}\sum_{i=1}^{N}\beta(\mathbf{a}_{k}|\hat{\overline{\mathbf{x}}}^{i}_{k})\geq\gamma_{k}$. If the inequality is satisfied, we proceed to step 3 (similar to standard PF), otherwise, we return to step 1 and redraw particles from the sampling density.\\
\textit{3) Weight computation:} Set $\hat{\widetilde{\mathbf{x}}}^{i}_{k}=\hat{\overline{\mathbf{x}}}^{i}_{k}$ and $\widetilde{\mathbf{y}}^{i}_{k}=\overline{\mathbf{y}}^{i}_{k}$ for $i=1,2,\hdots,N$. These particles estimate the prediction distribution $\pi_{k|k-1}$ as
    \par\noindent\small
    \begin{align*}
    \pi_{k|k-1}\approx\widetilde{\pi}^{N}_{k|k-1}(d\hat{\mathbf{x}}_{k},d\mathbf{y}_{k})\doteq\frac{1}{N}\sum_{i=1}^{N}\delta(\hat{\mathbf{x}}_{k}-\hat{\widetilde{\mathbf{x}}}^{i}_{k},\mathbf{y}_{k}-\widetilde{\mathbf{y}}^{i}_{k})d\hat{\mathbf{x}}_{k}d\mathbf{y}_{k}.
    \end{align*}
    \normalsize
    
    Since particles$\{\hat{\widetilde{\mathbf{x}}}^{i}_{k},\widetilde{\mathbf{y}}^{i}_{k}\}$ have equal weights due to resampling, using \eqref{eqn:optimal weights}, we compute the weights as $\widetilde{\omega}^{i}_{k}=\beta(\mathbf{a}_{k}|\hat{\widetilde{\mathbf{x}}}^{i}_{k})$ for $i=1,2,\hdots,N$ and normalize $\omega^{i}_{k}=\widetilde{\omega}^{i}_{k}/\sum_{j=1}^{N}\widetilde{\omega}^{j}_{k}$. Using $\{\hat{\widetilde{\mathbf{x}}}^{i}_{k},\widetilde{\mathbf{y}}^{i}_{k}\}$ with weights $\{\omega^{i}_{k}\}$, we obtain the approximate posterior distribution
    \par\noindent\small
    \begin{align*}
    \pi_{k|k}\approx\widetilde{\pi}^{N}_{k|k}(d\hat{\mathbf{x}}_{k},d\mathbf{y}_{k})\doteq\sum_{i=1}^{N}\omega^{i}_{k}\delta(\hat{\mathbf{x}}_{k}-\hat{\widetilde{\mathbf{x}}}^{i}_{k},\mathbf{y}_{k}-\widetilde{\mathbf{y}}^{i}_{k})d\hat{\mathbf{x}}_{k}d\mathbf{y}_{k}.
    \end{align*}
    \normalsize
\textit{4) Resampling:} Resample the particles by drawing $N$ independent particles as $(\hat{\mathbf{x}}^{i}_{k},\mathbf{y}^{i}_{k})\sim\widetilde{\pi}^{N}_{k|k}(d\hat{\mathbf{x}}_{k},d\mathbf{y}_{k})$. These uniformly weighted particles $\{\hat{\mathbf{x}}^{i}_{k},\mathbf{y}^{i}_{k}\}$ approximate $\pi_{k|k}$ as
    \par\noindent\small
    \begin{align*}
    \pi_{k|k}\approx\pi^{N}_{k|k}(d\hat{\mathbf{x}}_{k},d\mathbf{y}_{k})\doteq\frac{1}{N}\sum_{i=1}^{N}\delta(\hat{\mathbf{x}}_{k}-\hat{\mathbf{x}}^{i}_{k},\mathbf{y}_{k}-\mathbf{y}^{i}_{k})d\hat{\mathbf{x}}_{k}d\mathbf{y}_{k}.
    \end{align*}
    \normalsize
Here, similar to \cite{hu2008basic}, we have introduced an optional modification (step 2) for convenience of the convergence analysis in Section~\ref{sec:ipf convergence}. As mentioned earlier, the resampled particles $\{\mathbf{y}^{i}_{k}\}$ in step 4 are not needed for the next recursion. In general, the estimate of $\phi(\hat{\mathbf{x}},\mathbf{y})$ is computed prior to resampling for better accuracy. For instance, we compute defender's state estimate $\doublehat{\mathbf{x}}_{k}$ as $\doublehat{\mathbf{x}}_{k}=\sum_{i=1}^{N}\omega^{i}_{k}\hat{\widetilde{\mathbf{x}}}^{i}_{k}$.
\begin{remark}[Threshold $\gamma_{k}$ intuition]\label{remark:threshold}
    Note that the optimal inverse filter \eqref{eqn:optimal filter time update}-\eqref{eqn:optimal filter measurement update} exists if $\langle\pi_{k|k-1},\beta\rangle>0$. In I-PF, we approximate $\pi_{k|k-1}$ by $\widetilde{\pi}^{N}_{k|k-1}$ such that $\langle\pi_{k|k-1},\beta\rangle\approx\langle\widetilde{\pi}^{N}_{k|k-1},\beta\rangle=\frac{1}{N}\sum_{i=1}^{N}\beta(\mathbf{a}_{k}|\hat{\widetilde{x}}^{i}_{k})$. In step 2, we require $\langle\widetilde{\pi}^{N}_{k|k-1},\beta\rangle\geq\gamma_{k}$. Hence, this condition is motivated by the existence of the optimal filter and has been previously used as an indicator of divergence in PFs \cite{crisan2002survey,hu2008basic}. The threshold $\gamma_{k}$ must be chosen so that the inequality is satisfied for sufficiently large $N$ and in practice, modifies the PF algorithm only for small $N$. Theorem~\ref{thm:ipf convergence} further guarantees that the algorithm will not run into an infinite loop (in steps 1 and 2) provided that $\gamma_{k}$ is chosen small enough.
\end{remark}
%---------------------------------------------------------------------
\begin{figure}
  \centering
  \includegraphics[width = 1.0\columnwidth]{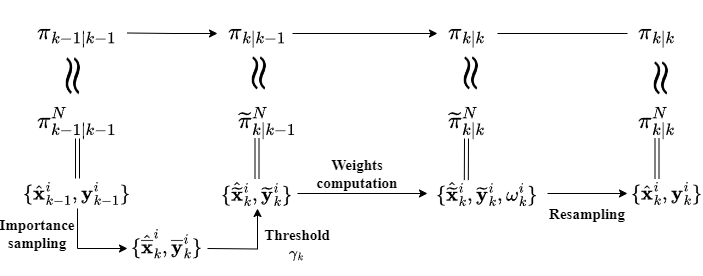}
  \caption{Graphical representation of posterior distributions in I-PF.}
 \label{fig:ipf schematic}
\end{figure}
%-----------------------------------------------------
%---------------------------------------------------------------
    \begin{table}[ht]
    \caption{I-PF compared with I-KF\cite{krishnamurthy2019how}}
    \label{tbl:differences_ikf}
    \centering
    \begin{tabular}{p{1.3cm}p{3cm}p{3cm}}
    \hline\noalign{\smallskip}
    Detail & I-KF & I-PF\\
    \noalign{\smallskip}
    \hline
    \noalign{\smallskip}
    System model & Linear, additive Gaussian noise & Non-linear, non-Gaussian\\
    Forward filter & KF & General forward filter $T(\cdot)$\\
    Posterior distribution & Gaussian & Empirical approximation of an arbitrary posterior\\
    Principle & Standard KF-like linear updates to compute $\doublehat{\mathbf{x}}_{k}$ & SIS and resampling to approximate $p(\hat{\mathbf{x}}_{k},\mathbf{y}_{k}|\mathbf{x}_{0:k},\mathbf{a}_{1:k})$\\
    \noalign{\smallskip}
    \hline\noalign{\smallskip}
    \end{tabular}
    \vspace{-10pt}
    \end{table}
 %--------------------------------------------------------------
 %------------------------------------------------------------
    \begin{table}[ht]
    \caption{I-PF compared with conventional PF}
    \label{tbl:differences_pf}
    \centering
    \begin{tabular}{p{1.3cm}p{3cm}p{3cm}}
    \hline\noalign{\smallskip}
    Detail & Conventional PF & I-PF\\
    \noalign{\smallskip}
    \hline
    \noalign{\smallskip}
    Inference & Estimate $\mathbf{x}_{k}$ given observations $\mathbf{y}_{1:k}$ & Estimate $\hat{\mathbf{x}}_{k}$ given true states $\mathbf{x}_{0:k}$ and observations $\mathbf{a}_{1:k}$\\
    Employing agent & Attacker & Defender, with a general forward filter $T(\cdot)$ at attacker's end\\
    Posterior distribution & Empirical approximation for $p(\mathbf{x}_{k}|\mathbf{y}_{1:k})$ & Empirical approximation for $p(\hat{\mathbf{x}}_{k},\mathbf{y}_{k}|\mathbf{x}_{0:k},\mathbf{a}_{1:k})$\\
    Sampling density & Transitional prior; sampling from optimal density is impractical & Optimal density; sampling from transitional prior is impractical\\
    Convergence guarantees & Bounded state transition $\mathcal{K}(\cdot)$ and observation $\rho(\cdot)$ & Bounded observations $\rho(\cdot)$ and $\beta(\cdot)$\\
    \noalign{\smallskip}
    \hline\noalign{\smallskip}
    \end{tabular}
    \vspace{-10pt}
    \end{table}
 %--------------------------------------------------------------

 For the proposed I-PF, under the assumption of known forward filter, we need to be able to sample from observation distribution $\rho(\cdot)$ and compute the density $\beta(\mathbf{a}_{k}|\hat{\mathbf{x}}_{k})$ at current observation $\mathbf{a}_{k}$. I-PF can handle non-Gaussian systems if these conditions are met. Fig.~\ref{fig:ipf schematic} provides a schematic illustration of posterior distributions and their approximations, while Tables~\ref{tbl:differences_ikf} and \ref{tbl:differences_pf} highlight the key differences of I-PF from I-KF\cite{krishnamurthy2019how} and the conventional PF, respectively.
\begin{remark}[I-PF's optimal importance density]\label{remark:IPF optimal density}
    For the considered defender-attacker dynamics, we are able to sample from the optimal density \eqref{eqn:optimal density} under the assumption of known forward filter. Another popular but suboptimal choice of $q(\cdot)$ is the transitional prior \cite{ristic2003beyond}, i.e., $q(\hat{\mathbf{x}}_{k},\mathbf{y}_{k}|\hat{\mathbf{x}}_{k-1},\mathbf{y}_{k-1},\mathbf{x}_{k},\mathbf{a}_{k})=p(\hat{\mathbf{x}}_{k},\mathbf{y}_{k}|\hat{\mathbf{x}}_{k-1},\mathbf{y}_{k-1})$ for I-PF. However, $p(\hat{\mathbf{x}}_{k},\mathbf{y}_{k}|\hat{\mathbf{x}}_{k-1},\mathbf{y}_{k-1})=p(\hat{\mathbf{x}}_{k}|\hat{\mathbf{x}}_{k-1},\mathbf{y}_{k})p(\mathbf{y}_{k}|\mathbf{y}_{k-1})$ such that to sample from the transitional prior, we need to sample from $p(\mathbf{y}_{k}|\mathbf{y}_{k-1})$. Hence, contrary to the standard PF, it is easier to sample from the I-PF's optimal density \eqref{eqn:optimal density} than the transitional prior. This is possible because of the perfect knowledge of actual state $\mathbf{x}_{k}$ available to the defender such that we can directly sample from $p(\mathbf{y}_{k}|\mathbf{x}_{k})$.
\end{remark}

\vspace{-8pt}
\subsection{I-PF variants}\label{subsec:IPF variants}
As in the case of standard PF, resampling in I-PF reduces degeneracy of particles over time, but introduces sample impoverishment \cite{arulampalam2002tutorial}. Since the particles with large weights are statistically selected many times, there is a loss of diversity among particles, which may also lead to collapse to a single point in case of small system noises. Different techniques like Markov chain MC move step \cite{carlin1992monte} and regularization \cite{musso2001improving} have been proposed to address sample impoverishment. The I-PF developed here is a basic inverse filter based on the SIS and resampling techniques. Different choices of sampling density $q(\cdot)$ and/or modification of the resampling step lead to different variants of PFs. Similar modifications can also be introduced to the basic I-PF to obtain suitable variants for different applications. For instance, in auxiliary PF \cite{pitt1999filtering}, the previous particles are resampled conditioned on the current measurement before importance sampling such that the particles are most likely close to the current true state. On the other hand, regularized PF \cite{musso2001improving} considers a kernel density $K_{h}(\cdot)$ to resample from a continuous distribution instead of a discrete one, i.e., $(\hat{\mathbf{x}}^{i}_{k},\mathbf{y}^{i}_{k})\sim\sum_{i=1}^{N}\omega^{i}_{k}K_{h}(\hat{\mathbf{x}}_{k}-\hat{\widetilde{\mathbf{x}}}^{i}_{k},\widetilde{\mathbf{y}}_{k}-\mathbf{y}^{i}_{k})$ for $i=1,2,\hdots,N$.

\vspace{-8pt}
\section{Convergence guarantees}
\label{sec:ipf convergence}
Several guarantees for convergence of the standard PF to the optimal filter's posterior have been provided in the literature \cite{moral2004feynman,crisan2002survey}. In \cite{le2004stability}, a Hilbert projective metric is considered to study the optimal filter's stability, which is then used to derive uniform convergence conditions for PFs. The survey by \cite{crisan2002survey} showed almost sure convergence assuming a Feller transition kernel and a bounded, continuous, strictly positive observation likelihood. Other approaches include central limit theorems  \cite{del1999central} and large deviations \cite{crisan1999large,del1998large}. However, all these prior works assume the (estimated) function $\phi(\mathbf{x})$ of the underlying state $\mathbf{x}$ to be bounded and hence, exclude the state estimate itself, i.e., $\phi(\mathbf{x})=\mathbf{x}$. Recently, \cite{hu2008basic} have addressed the general case of the unbounded function $\phi$ and proved PF's convergence in $L^{4}$-sense under some mild assumptions on the rate of increase of $\phi$. An extended $L^{p}$-convergence result has also been obtained using a Rosenthal-type inequality in \cite{hu2011general}.

In the following, we consider the $L^{4}$-approach of \cite{hu2008basic} to derive the conditions for our I-PF's convergence to the optimal filter \eqref{eqn:optimal filter time update}-\eqref{eqn:optimal filter measurement update}. Particularly, we show that for a given time $k$ and given observations $\mathbf{a}_{1:k}$ and known true states $\mathbf{x}_{0:k}$, the I-PF's $\langle\pi^{N}_{k|k},\phi\rangle$ converges to the optimal filter's $\langle\pi_{k|k},\phi\rangle$ as the number of particles $N$ increases. Note that for given $\{\mathbf{a}_{1:k},\mathbf{x}_{0:k}\}$, the optimal distribution $\pi_{k|k}$ is deterministic, but $\pi^{N}_{k|k}$ is random because of the randomly generated particles. Hence, all the stochastic expectations $\mathbb{E}$ and almost sure convergence are with respect to these random particles.

We assume that the system model \eqref{eqn:state x}-\eqref{eqn:observation a} and estimated function $\phi$ satisfy the following conditions.\\
\textbf{A1.} For given $\mathbf{x}_{0:s}$ and $\mathbf{a}_{1:s}$ for $s=1,2,\hdots,k$, $\langle\pi_{s|s-1},\beta\rangle>0$ and there exits $\{\gamma_{s}\}_{1\leq s\leq k}$ such that $0<\gamma_{s}<\langle\pi_{s|s-1},\beta\rangle$.\\
\textbf{A2.} The observation densities are bounded, i.e., $\beta(\mathbf{a}_{s}|\hat{\mathbf{x}}_{s})<\infty$ and $\rho(\mathbf{y}_{s}|\mathbf{x}_{s})<\infty$ for $s=1,2,\hdots,k$.\\
\textbf{A3.} The function $\phi$ satisfies $\textrm{sup}_{(\hat{\mathbf{x}}_{s},\mathbf{y}_{s})} |\phi(\hat{\mathbf{x}}_{s},\mathbf{y}_{s})|^{4}\beta(\mathbf{a}_{s}|\hat{\mathbf{x}}_{s})<C(\mathbf{x}_{0:s},\mathbf{a}_{1:s})$ for given $\mathbf{x}_{0:s},\mathbf{a}_{1:s}, s=1,2,\hdots,k$. Here, $C(\mathbf{x}_{0:s},\mathbf{a}_{1:s})$ is a finite constant which may depend on $\mathbf{x}_{0:s},$ and $\mathbf{a}_{1:s}$.\\
Recall that $\{\gamma_{s}\}$ are the thresholds introduced in the modification step (step 2) of the I-PF algorithm in Section~\ref{subsec:ipf formulation}. As discussed in Remark~\ref{remark:threshold}, assumption \textbf{A1} is related to the existence of the optimal filter and divergence of PFs. Intuitively, assumption \textbf{A3} states that the conditional observation density $\beta$ must decrease at a rate faster than the function $\phi$ increases. Furthermore, \textbf{A2} and \textbf{A3} imply that the conditional fourth moment of $\phi$ is bounded, i.e., $\langle\pi_{s|s},|\phi|^{4}\rangle<\infty$ \cite{hu2008basic}. Interestingly, while the convergence conditions for PF in \cite{hu2008basic} also assume a bounded transition kernel $\mathcal{K}(\cdot)$, our I-PF only requires bounded observation densities as in \textbf{A2}. Finally, the I-PF convergence is provided in the following theorem.
\begin{theorem}[I-PF convergence]\label{thm:ipf convergence}
Consider the I-PF developed in Section~\ref{subsec:ipf formulation} (including the modification step). If the assumptions \textbf{A1}-\textbf{A3} are satisfied, then the following hold:\\
\textit{1)} For sufficiently large $N$, the algorithm will not run into an infinite loop in steps $1-2$.\\
\textit{2)} For any $\phi$ satisfying \textbf{A3}, there exists a constant $C_{k|k}$, independent of $N$ such that
    \par\noindent\small
    \begin{align}
        \mathbb{E}|\langle\pi^{N}_{k|k},\phi\rangle-\langle\pi_{k|k},\phi\rangle|^{4}\leq C_{k|k}\frac{\|\phi\|^{4}_{k,4}}{N^{2}},\label{eqn:ipf converge}
    \end{align}
    \normalsize
    where $\|\phi\|_{k,4}\doteq\textrm{max}\{1,\textrm{max}_{\;0\leq s\leq k}\langle\pi_{s|s},|\phi|^{4}\rangle^{1/4}\}$ and $\pi^{N}_{k|k}$ is generated by the I-PF algorithm.
\end{theorem}
\begin{proof}
    See Appendix~\ref{app:ipf convergence}.
\end{proof}
As a consequence of Theorem~\ref{thm:ipf convergence}, the following corollary can be obtained trivially using the Borel-Cantelli lemma as in \cite[Proposition~7.2.3(a)]{athreya2006measure}.
\begin{corollary}\label{cor:almost sure convergence}
If \textbf{A1}-\textbf{A3} holds, then for any $\phi$ satisfying \textbf{A3}, we have $\lim_{N\to\infty}\langle\pi^{N}_{k|k},\phi\rangle=\langle\pi_{k|k},\phi\rangle$ in the almost sure convergence sense.
\end{corollary}
\begin{remark}[Dependence on state dimension]
    From \eqref{eqn:ipf converge}, we observe that the I-PF's convergence rate does not depend on the state dimension $n_{x}$ and hence, I-PF does not suffer from the curse of dimensionality. However, for a given/desired bound on the error, the required number of particles $N$ depends on constant $C_{k|k}$, which can depend on $n_{x}$.
\end{remark}
\begin{remark}[Bound on $C_{k|k}$]
    In general, without any additional assumptions, we cannot guarantee that $C_{k|k}$ will not increase over time. In particular, if the optimal filter does not `forget' its initial condition, the approximation errors accumulate over time such that $C_{k|k}$ increases \cite{crisan2002survey}. Hence, the required number of particles $N$ also increases proportionally for the error to remain within a given bound. However, prior works \cite{del2001stability,le2004stability} provide additional conditions on the system model to ensure the optimal filter mixes quickly (forgets its initial condition) and $C_{k|k}$ does not increase with time.
\end{remark}

\vspace{-8pt}
\section{Unknown system dynamics}\label{sec:unknown}
So far, we assumed that both the attacker and the defender have perfect system information. Nevertheless, in real-world scenarios, the agents employing the stochastic filters may not have any prior knowledge about the system model. The attacker may not be aware of the defender's state evolution process \eqref{eqn:state x}. On the other hand, the defender may not know the attacker's action strategy such that \eqref{eqn:observation a} is not available. Further, the attacker's forward filter may be unknown and assuming a simple forward EKF/UKF may be inefficient. To this end, in the following, we present a differentiable I-PF to learn the state estimates and the model parameters.

Differentiable PFs (DPFs) construct the system dynamics and the proposal distributions using learning networks and optimize them using gradient descent. However, major challenges in developing DPFs are the non-differentiable importance sampling and resampling steps. Sampling from a proposal distribution is not differentiable because of the absence of explicit dependency between the sampled particles and the distribution parameters. On the other hand, the discrete nature of multinomial resampling makes it inherently non-differentiable, i.e., a small change in input weights can lead to abrupt changes in the resampling output. Additionally, the resampled particles are equally weighted, which is a constant such that the gradients are always zero. Hence, DPFs employ reparameterization-based differentiable sampling, and various differentiable resampling techniques \cite{chen2023overview,karkus2018particle,corenflos2021differentiable,zhu2020towards}. Another important factor affecting DPFs' performance is the loss function minimized in gradient descent to optimize the parameters.

Here, we discuss how this differentiable framework can be integrated into our I-PF to handle unknown dynamics in inverse filtering. Following I-PF's formulation in Section~\ref{subsec:ipf formulation}, differentiable I-PF considers the joint conditional density $p(\hat{\mathbf{x}}_{k},\mathbf{y}_{k}|\mathbf{x}_{0:k},\mathbf{a}_{1:k})$. The formulation then closely follows from standard DPF methods, and hence, we only summarize them here. We refer the readers to \cite{chen2023overview} (and the references therein) for further details.

\subsubsection{Proposal distributions} The simplest choice of proposal distribution in PFs is the system's state evolution. In DPFs, the sampled particle from this state evolution is computed as a function of the previous particle and an additional noise term. The corresponding function is parameterized by an unknown parameter $\theta$ and is differentiable with respect to both the previous particle and the noise term. Similarly, in differentiable I-PF, we can consider a differentiable model for the optimal importance sampling density \eqref{eqn:optimal density}. For instance, we can model $p(\hat{\mathbf{x}}_{k}|\hat{\mathbf{x}}_{k-1},\mathbf{y}_{k})p(\mathbf{y}_{k}|\mathbf{x}_{k})$ as a jointly Gaussian distribution $\mathcal{N}(\hat{\mathbf{x}},\mathbf{y}_{k};\mu_{\theta}(\hat{\mathbf{x}}_{k-1},\mathbf{x}_{k}),\bm{\Sigma}_{\theta})$. In this case, the sampled particles $(\hat{\widetilde{\mathbf{x}}}^{i}_{k},\widetilde{\mathbf{y}}^{i}_{k})$ are obtained by adding zero-mean Gaussian noise of covariance $\bm{\Sigma}_{\theta}$ to the mean $\mu_{\theta}(\hat{\mathbf{x}}^{i}_{k-1},\mathbf{x}_{k})$ where $\hat{\mathbf{x}}^{i}_{k-1}$ is the (state estimate) particle from previous time instant and $\mathbf{x}_{k}$ is the defender's true state known perfectly. The function $\mu_{\theta}(\cdot)$ is a differentiable function of $\hat{\mathbf{x}}_{k-1}$ while $\bm{\Sigma}_{\theta}$ can be designed manually \cite{karkus2018particle,wen2021end} or parameterized for learning \cite{kloss2021train}.
    
Alternatively, the differential sampling technique based on normalizing flows\cite{chen2021differentiable} draws samples from simple distributions (like Gaussian or uniform) and transforms them into arbitrary distributions through a series of invertible mappings under some mild conditions. Note that even though \eqref{eqn:optimal density} minimizes the variance of weights, the information from current observation $\mathbf{a}_{k}$ is not utilized. In general, constructing proposal distributions with observation $\mathbf{a}_{k}$ provides samples that are closer to the true posterior and more uniformly weighted. Both the Gaussian model and normalising flow-based differential samplings can be generalized to include observation $\mathbf{a}_{k}$ in the sampling distribution \cite{chen2021differentiable,karkus2021differentiable}. Note that if the true state $\mathbf{x}_{k}$ information is not available to the defender, we can directly model $q^{*}(\hat{\mathbf{x}}_{k},\mathbf{y}_{k}|\hat{\mathbf{x}}_{k-1},\mathbf{y}_{k-1},\mathbf{a}_{k})$ from \eqref{eqn:optimal density} as a differentiable function/ distribution of $\hat{\mathbf{x}}^{i}_{k-1},\mathbf{y}^{i}_{k-1}$ and $\mathbf{a}_{k}$, i.e., the sampled particles and available observation.

\subsubsection{Observation models} In differentiable I-PF, instead of \eqref{eqn:observation a}, we consider a parameterized observation model as $p_{\theta}(\mathbf{a}_{k}|\hat{\mathbf{x}}_{k})\propto l_{\theta}(\mathbf{a}_{k},\hat{\mathbf{x}}_{k})$ where $l_{\theta}(\cdot)$ is a differentiable function with respect to $\mathbf{a}_{k}$ and $\hat{\mathbf{x}}_{k}$. To this end, we can consider known distribution with learnable parameters \cite{corenflos2021differentiable} or approximate $p_{\theta}(\mathbf{a}_{k}|\hat{\mathbf{x}}_{k})$ using a scalar function learned from neural network (NN) \cite{karkus2018particle}. The observations can also be mapped to an NN-based feature space as $f_{k}=F_{\theta}(\mathbf{a}_{k})$ such that $l_{\theta}(\mathbf{a}_{k},\hat{\mathbf{x}}_{k})=h_{\theta}(f_{k},\hat{\mathbf{x}}_{k})$. In \cite{wen2021end} and \cite{chen2021differentiable}, NNs are used to extract both observation and state features to measure similarity/discrepancy using user-defined metrics. Alternatively, conditional normalizing flows can also be employed \cite{chen2022conditional}.

\subsubsection{Differentiable resampling} Soft resampling, optimal transport (OT)-based resampling, and particle transformer-based resampling are popular differentiable techniques employed in DPFs and can be readily applied to differentiable I-PF. Soft resampling \cite{karkus2018particle} aims to generate non-zero gradients by modifying the importance weights by a factor $\lambda$ as $\overline{\omega}^{i}_{k}=\lambda\omega^{i}_{k}+(1-\lambda)1/N$ where $N$ is the total number of particles. However, the particles are selected using a multinomial distribution, and hence, the outputs still change abruptly. Soft resampling can be viewed as a linear interpolation between the multinomial distribution of original weights and one with equal weights. Contrarily, resampling using entropy-regularized OT \cite{corenflos2021differentiable} is fully differentiable. In particular, in differentiable I-PF, OT provides a map between the equally weighted empirical distribution $\frac{1}{N}\sum_{i=1}^{N}\delta(\hat{\mathbf{x}}_{k}-\hat{\mathbf{x}}^{i}_{k},\mathbf{y}_{k}-\mathbf{y}^{i}_{k})$ and the target empirical distribution $\sum_{i=1}^{N}\omega^{i}_{k}\delta(\hat{\mathbf{x}}_{k}-\hat{\widetilde{\mathbf{x}}}^{i}_{k},\mathbf{y}_{k}-\widetilde{\mathbf{y}}^{i}_{k})$. Particle transformers \cite{zhu2020towards} are permutation-invariant and scale-equivalent NNs that take weighted particles $\{\omega^{i}_{k},(\hat{\widetilde{\mathbf{x}}}^{i}_{k},\widetilde{\mathbf{y}}^{i}_{k})\}_{1\leq i\leq N}$ as inputs and output resampled particles $\{(\hat{\mathbf{x}}^{i}_{k},\mathbf{y}^{i}_{k})\}_{1\leq i\leq N}$ with equal weights. Particle transformers perform differentiable resampling but require pre-training.

\subsubsection{Loss functions and training} Following the standard DPFs, differentiable I-PFs can be trained using supervised losses like root mean squared error (RMSE) or negative state likelihood when the ground truth, i.e., attacker's state estimate $\hat{\mathbf{x}}_{k}$ and observation $\mathbf{y}_{k}$ are available \cite{karkus2018particle,chen2021differentiable,corenflos2021differentiable}. Semi-supervised losses like marginal observation likelihoods are helpful when unlabelled data is abundant but access to labels is limited \cite{wen2021end}. Alternatively, variational inference optimizes evidence lower bound (ELBO) instead of likelihood to learn the model and proposal distribution simultaneously \cite{maddison2017filtering,le2018auto}. With these objectives, DPFs and hence differentiable I-PFs, can be trained end-to-end \cite{karkus2018particle,wen2021end} or individually \cite{karkus2021differentiable,lee2020multimodal}. In end-to-end training, all components of the filter are jointly trained via gradient descent to minimize an overall loss function. Contrarily, in individual training, various components are first pre-trained independently and then fine-tuned jointly for a task-specific objective.

The idea of DPF is to implement the traditional PF in a data-adaptive and differentiable manner, where the system models are defined using learning networks whose parameters are learned via backpropagation and gradient descent, including differentiation through the inference algorithm itself. Rather than modeling a generic system, DPFs learn an optimized model for inference using PF. In fact, \cite{jonschkowski2018differentiable} interpreted DPF as a recurrent NN (RNN) that leverages the recursive state estimation framework to enhance data efficiency and generalization. In this context, the methods described in points 1-3 outline ways to implement our I-PF such that gradients can be computed through differentiation. The choice of architecture depends on the application and input type. For instance, \cite{ma2020particle,karkus2018particle} used RNNs to model the hidden state's relative motion over one time-step and to extract features from observations, while \cite{lee2020multimodal} encoded image inputs using a convolutional NN (CNN). For high-dimensional observations, encoder networks are often used to learn compact representations\cite{wen2021end}. In \cite{karkus2021differentiable}, the transition and observation models were built using CNNs with spatial transformers for differentiable mapping. Additionally, normalizing flows\cite{papamakarios2021normalizing} and particle transformers\cite{zhu2020towards} are themselves novel architectures developed for constructing probability distributions and differentiable resampling, respectively. With these components, the differentiable I-PF can be trained like standard networks by minimizing task-specific or likelihood-based loss functions (discussed in point 4) via gradient descent. In \cite{wen2021end}, both Adam and RMSProp optimizers were applied with heuristic learning parameters.

\begin{remark}[Trajectory Function of Time]\label{remark:TFoT}
    As an alternative to DPFs, the recent trajectory function of time (T-FoT) fitting enables tracking smooth deterministic target trajectories without prior information\cite{li2018joint,li2016fitting,zhou2021target,li2023target}. Specifically, \cite{li2018joint,zhou2021target} model the unknown state process \eqref{eqn:state x} as a continuous function, i.e., $\mathbf{x}_{t}=f(t,\mathbf{c}_{i})$ where $t$ denotes the continuous time and $\{\mathbf{c}_{i}\}$ are fitting coefficients, transforming the state estimation problem into online curve fitting. Unlike Bayesian methods, it avoids temporal independence assumption (i.e., $\mathbf{x}_{k}$ is independent of $\mathbf{x}_{0:k-2}$ given $\mathbf{x}_{k-1}$) and provides richer insights compared to discrete-time point estimates. However, the T-FoT approach is less suitable for inverse filtering, wherein the attacker's inference process $\{\hat{\mathbf{X}}_{k}\}_{k\geq 0}$ is inherently non-smooth, stochastic, and influenced by the defender's maneuvers (via true state $\mathbf{x}_{k}$). Moreover, multi-dimensional regression being intractable \cite{li2016fitting}, T-FoT relies on conditional independence among state dimensions and physical relationships to infer unobserved variables (e.g., velocity from position), which do not hold for the attacker's state estimates. Consequently, T-FoT methods are currently inadequate to estimate the multi-dimensional outputs of a forward filter in these scenarios.
\end{remark}

\vspace{-5pt}
\section{Numerical Experiments}
\label{sec:simulation}
We demonstrate the estimation performance of the proposed filters compared to I-EKF \cite{singh2022inverse_part1} and I-UKF \cite{singh2023inverse_ukf} using three different example systems. Besides the estimation error, we also consider RCRLB \cite{tichavsky1998posterior} and NCI \cite{li2001practical} as performance metrics in Section~\ref{subsec:nonlinear example}. RCRLB provides a lower bound on MSE for discrete-time filtering, with simplified closed-form recursions for RCRLB computation for non-linear systems with additive Gaussian noises provided in  \cite{xiong2006performance_ukf}. On the other hand, NCI is a credibility measure of how statistically close the estimated error covariance of an estimator is to the actual MSE matrix. In Section~\ref{subsec:bearing tracking}, we focus on estimation error and time-complexity of the proposed algorithm, while Section~\ref{subsec:non gauss} considers a non-Gaussian time-series. Note that in practice, a non-linear filter's performance also depends on the system itself. Choosing a suitable filter for any application usually involves a trade-off between estimation accuracy and computational efforts \cite{li2017approximate}. The same argument holds for the non-linear inverse filters.

Throughout all experiments, for simplicity, we choose EKF as the forward filter $T(\cdot)$ in I-PF, regardless of the actual forward filter, unless mentioned otherwise. Note that I-EKF and I-UKF also assume a forward EKF and UKF, respectively\cite{singh2022inverse_part1,singh2023inverse_ukf}. All forward filters are initialized with the same initial distribution. In particular, if $\mathcal{N}(\mathbf{x}_{0};\hat{\mathbf{x}}_{0},\bm{\Sigma}_{0})$ is the forward filters' initial distribution, then we initialize forward EKF/UKF with initial state $\hat{\mathbf{x}}_{0}$ and initial covariance matrix $\bm{\Sigma}_{0}$ while forward PF consider independent samples drawn from the initial distribution. All inverse filters are initialized similarly.
%---------------------------------------------------------------------
\begin{figure}
  \centering
  \includegraphics[width = 1.0\columnwidth]{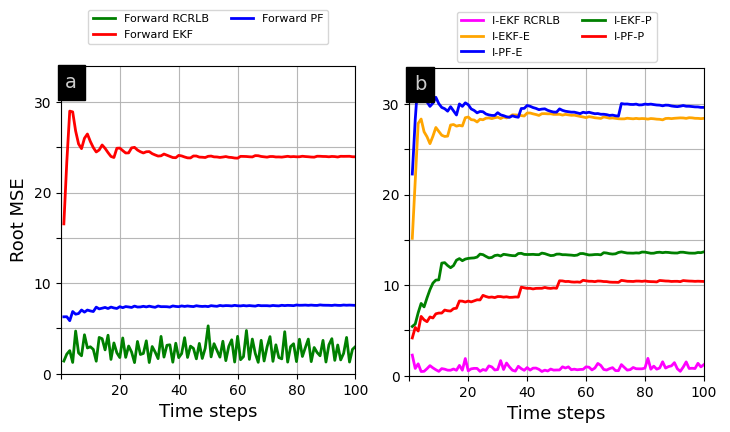}
  \caption{Time-averaged RMSE and RCRLB for (a) forward PF and EKF, and (b) I-PF and I-EKF, including mismatched forward filter cases, for non-linear system example.\vspace{-12pt}}
 \label{fig:nonlinear rmse}
\end{figure}
%-----------------------------------------------------

\vspace{-8pt}
\subsection{Illustrative non-linear system}
\label{subsec:nonlinear example}
Consider a non-linear system\cite{arulampalam2002tutorial}
\par\noindent\small
\begin{align*}
&x_{k+1}=\frac{x_{k}}{2}+\frac{25x_{k}}{1+x_{k}^{2}}+8\cos{(1.2k)}+\textrm{w}_{k},
\end{align*}
\normalsize
with observations $y_{k}=x_{k}^{2}/20+v_{k}$ and $a_{k}=\hat{x}^{2}_{k}/10+\epsilon_{k}$ where $\textrm{w}_{k}\sim\mathcal{N}(0,10)$, $v_{k}\sim\mathcal{N}(0,1)$ and $\epsilon_{k}\sim\mathcal{N}(0,5)$. We set the number of particles for forward PF as $25$ while that for I-PF as $50$. The initial distribution for forward and inverse filters were $\mathcal{N}(0,5)$ and $\mathcal{N}(0,10)$, respectively. Also, the function $T(\cdot)$ is initialized with distribution $\mathcal{N}(0,10)$ in I-PF.

Fig.~\ref{fig:nonlinear rmse} shows the time-averaged RMSE for the forward and inverse filters, averaged over $250$ runs. The I-PF's error in estimating state estimate $\hat{\mathbf{x}}_{k}$ when the attacker's actual forward filter is EKF is labeled as I-PF-E. The other notations in Fig.~\ref{fig:nonlinear rmse} and also, in further experiments, are similarly defined. In Fig.~\ref{fig:nonlinear rmse}, we also include the corresponding RCRLBs. Note that RCRLB for forward PF cannot be derived in closed form because of the lack of an explicit state-transition function and hence, omitted here. Figure \ref{fig:nonlinear rmse}a illustrates that the forward PF outperforms the forward EKF. However, in Figure \ref{fig:nonlinear rmse}b, the errors for I-EKF and I-PF are comparable when estimating forward EKF's $\hat{\mathbf{x}}_{k}$, in the I-EKF-E and I-PF-E scenarios, respectively. Note that, when the forward EKF assumption is violated, I-PF surpasses I-EKF, as seen in the I-EKF-P and I-PF-P cases, making I-PF more accurate for the mismatched forward filter scenario. Figure \ref{fig:nci bearing}a depicts the NCI for various forward and inverse filters. Here, the forward PF achieves near-perfect NCI (close to $0$), indicating its higher credibility. Similarly, I-PF-E and I-PF-P demonstrate significantly lower NCI (in magnitude) than I-EKF-E and I-EKF-P, with I-PF-P achieving perfect NCI once again. On the other hand, both forward and inverse EKFs tend to be overly pessimistic, with their estimated covariance exceeding the actual MSE matrix.

\vspace{-8pt}
\subsection{Bearing only tracking}
\label{subsec:bearing tracking}
Consider a moving sensor tracking a target moving along x-axis using bearing measurements only. Denote $p_{k}$ and $v_{k}$ as the target's position (in m) and velocity (in $\textrm{m}/\textrm{sec}$), respectively, at the $k$-th time instant. The sensor's position is $(s^{x}_{k},s^{y}_{k})$ with $s^{x}_{k}=4k+\Delta s^{x}_{k}$ and $s^{y}_{k}=20+\Delta s^{y}_{k}$ where $\Delta s^{x}_{k}$ and $\Delta s^{y}_{k}$ denote perturbations distributed as $\mathcal{N}(0,1)$. Then, the system model is \cite{lin2002comparison}
\par\noindent\small
\begin{align*}
&\mathbf{x}_{k+1}\doteq\begin{bmatrix}p_{k+1}\\v_{k+1}\end{bmatrix}=\begin{bmatrix}1&T\\0&1\end{bmatrix}\begin{bmatrix}p_{k}\\v_{k}\end{bmatrix}+\begin{bmatrix}T^{2}/2\\T\end{bmatrix}\textrm{w}_{k},\\
&y_{k}=\tan^{-1}\left(\frac{s^{y}_{k}}{p_{k}-s^{x}_{k}}\right)+v_{k},\;a_{k}=\tan^{-1}\left(\frac{s^{y}_{k}}{\hat{p}_{k}-s^{x}_{k}}\right)+\epsilon_{k},
\end{align*}
\normalsize
where $\textrm{w}_{k}\sim\mathcal{N}(0,0.01)$, $v_{k}\sim\mathcal{N}(0,(3^{\circ})^{2})$ and $\epsilon_{k}\sim\mathcal{N}(0,(5^{\circ})^{2})$ with $T=1$ sec. The initial state $\mathbf{x}_{0}$ was $[80,1]^{T}$. The forward filters were initialized with $\mathcal{N}(\hat{\mathbf{x}}_{0},\bm{\Sigma}_{0})$ where $\hat{\mathbf{x}}_{0}=[20/\tan^{-1}(y_{1}),0]^{T}$ and $\bm{\Sigma}_{0}=\textrm{diag}(16,1)$. The inverse filters were initialized with $\mathcal{N}(\mathbf{x}_{0},\mathbf{I}_{2})$. For both forward and inverse PFs, the number of particles was set to $100$. The recursion function $T(\cdot)$ was initialized with $\mathcal{N}(\mathbf{x}_{0},\mathbf{I}_{2})$. EKF's implementation for the considered system is provided in \cite{bar2004estimation}.
%---------------------------------------------------------------------
\begin{figure}
  \centering
  \includegraphics[width = 1.0\columnwidth]{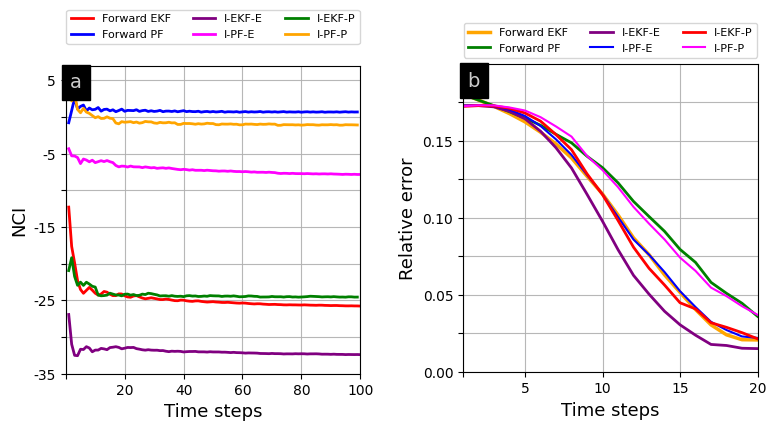}
  \caption{(a) NCI for forward and inverse filters for non-linear system example; and (b) relative error for forward and inverse PF and EKF for bearing-only tracking system.\vspace{-12pt}}
 \label{fig:nci bearing}
\end{figure}
%-----------------------------------------------------
 
Fig.~\ref{fig:nci bearing}b shows the relative error in position for the forward and inverse filters, averaged over $100$ runs. In the bearing tracking example, the forward PF shows a slightly higher error than the forward EKF, except that EKF requires Jacobian computations. The same pattern is observed for the inverse filters, where I-PF-E and I-PF-P exhibit higher errors than I-EKF-E and I-EKF-P, respectively. The superior accuracy of both forward and inverse EKFs in this system is attributed to the Gaussian posterior assumption inherent in these filters. Although the forward PF and I-PF approximate arbitrary posterior distributions empirically, the forward EKF and I-EKF assume a Gaussian posterior, estimating its mean and covariance through Taylor series linearization. Fig.~\ref{fig:nci bearing}b suggests that for this system, the Gaussian posterior assumption actually results in better accuracy compared to the arbitrary distribution approach employed by PF. In fact, \cite{kotecha2003gaussian} recently demonstrated that the Gaussian posterior assumption enhances the accuracy of standard PFs in bearing-only tracking scenarios. Table \ref{tbl:bearing} provides the total run time for 20 time steps (in a single Monte Carlo run) for different forward and inverse filters. As expected, the forward PF has a longer run time than the forward EKF, and this increases as the number of particles $N$ rises. Similarly, I-PF-E and I-PF-P exhibit higher time complexity compared to I-EKF-E and I-EKF-P, respectively. The run time for I-PF is also longer than that of the forward PF, as it requires computing function $T(\cdot)$ for each particle at every recursion. Interestingly, for all $N$, both I-EKF and I-PF have slightly shorter run times when estimating forward PF's $\hat{\mathbf{x}}_{k}$ compared to when estimating forward EKF's $\hat{\mathbf{x}}_{k}$.
%-----------------------------------------------------------------
    \begin{table}
    \caption{Run time (in seconds) for different filters}% for bearing only tracking system.}
    \label{tbl:bearing}
    \centering
    \begin{tabular}{p{2.0cm}p{1.5cm}p{1.5cm}p{1.5cm}}
    \hline\noalign{\smallskip}
    Filter & $N=100$ & $N=250$ & $N=500$\\
    \noalign{\smallskip}
    \hline
    \noalign{\smallskip}
    Forward EKF & 0.0113 & 0.0119 & 0.0094\\
    Forward PF & 0.0191 & 0.0291 & 0.0389\\
    I-EKF-E & 0.0096 & 0.0104 & 0.0115\\
    I-EKF-P & 0.0012 & 0.0009 & 0.0010\\
    I-PF-E & 0.0415 & 0.0856 & 0.1588\\
    I-PF-P & 0.0314 & 0.0761 & 0.1456\\
    \noalign{\smallskip}
    \hline\noalign{\smallskip}
    \end{tabular}
    \vspace{-12pt}
\end{table}
 %----------------------------------------------------------------
\vspace{-8pt}
\subsection{Non-Gaussian system}
\label{subsec:non gauss}
Consider the one-dimensional non-Gaussian time series from \cite[Sec.~6]{van2000unscented_pf} and non-stationary observations for attacker and defender as follows: 
\par\noindent\small
\begin{align*}
&x_{k+1}=1+\sin{(\Omega\pi k)}+\phi_{1}x_{k}+\textrm{w}_{k},\\
&y_{k}=\begin{cases}
    \phi_{2}x^{2}_{k}+v_{k}, & k\leq 30\\
    \phi_{3}x_{k}-2+v_{k}, & k>30
\end{cases},\\
&a_{k}=\begin{cases}
    \phi_{4}\hat{x}^{2}_{k}+\epsilon_{k}, & k\leq 35\\
    \phi_{5}\hat{x}_{k}-1+\epsilon_{k}, & k>35
\end{cases},
\end{align*}
\normalsize
where parameters $\Omega=4e-1$, $\phi_{1}=0.5$, $\phi_{2}=0.2$, $\phi_{3}=0.5$, $\phi_{4}=0.02$ and $\phi_{5}=0.05$. The process noise $\textrm{w}_{k}$ follows a non-Gaussian distribution $\textrm{w}_{k}\sim\textrm{Gamma}(3,2)$, while observation noises are distributed as $v_{k}\sim\mathcal{N}(0,10^{-5})$ and $\epsilon_{k}\sim\mathcal{N}(0,1)$. As shown in \cite[Sec.~6]{van2000unscented_pf}, the unscented transform improves PF's accuracy for this system. Accordingly, we consider EKF, UKF, and Unscented PF (UPF)\cite{van2000unscented_pf} as forward filters, with EKF and UKF assuming $\textrm{w}_{k}$ is Gaussian with mean $6$ and variance $6$. Similarly, our I-PF uses UKF as the forward filter function $T(\cdot)$. The unscented transform's scaling parameter $\lambda$ is set to $0.5$, $1$, $2$, and $2$ for forward UKF, I-UKF, forward UPF, and I-PF's $T(\cdot)$, respectively. Forward UPF and I-PF consider $200$ and $300$ particles, respectively. The initial state is $x_{0}=0$, with forward and inverse filters initialized with $\mathcal{N}(0,1)$ and $\mathcal{N}(0,5)$, respectively.

Fig.~\ref{fig:non gauss} shows the time-averaged RMSE for forward and inverse filters, averaged over $100$ runs. For the considered non-Gaussian system, the forward UPF outperforms both forward EKF and UKF. In Fig.~\ref{fig:non gauss}a, I-EKF-E achieves the lowest error when estimating forward EKF's estimates, as it correctly assumes the forward filter. Similarly, in Fig.~\ref{fig:non gauss}b, I-UKF-U demonstrates the highest accuracy when estimating forward UKF's estimates, although I-PF-U also significantly outperforms I-EKF-U. Interestingly, I-UKF-P and I-PF-P exhibit similar errors, both lower than I-EKF-P. The improved performance of I-UKF in this system is attributed to the effective approximation of the unscented transform. However, note that here we utilized the basic I-PF from Section~\ref{subsec:IPF recursion}. Several modifications discussed in Section~\ref{subsec:IPF variants} can further enhance I-PF's performance.

 %---------------------------------------------------------------------
\begin{figure}
  \centering
  \includegraphics[width = 1.0\columnwidth]{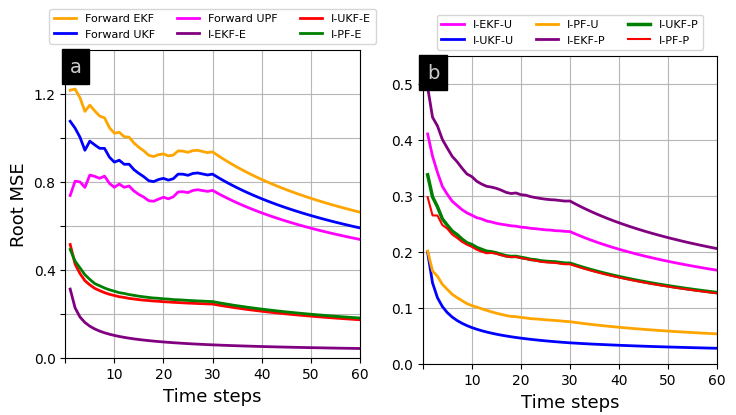}
  \caption{Time-averaged RMSE for (a) forward filters and inverse filters estimating the forward EKF estimates, and (b) inverse filters estimating the forward UKF and UPF estimates in a non-Gaussian system example.\vspace{-12pt}}
 \label{fig:non gauss}
\end{figure}
%-----------------------------------------------------

\vspace{-8pt}
\section{Summary}
\label{sec:summary}
We studied the inverse filtering problem in counter-adversarial systems and developed I-PF based on the SIS and resampling techniques. Unlike prior inverse filters, the proposed I-PF considers general attacker-defender dynamics, including non-Gaussian systems. Furthermore, we proved that under mild assumptions on the system model, I-PF converges to the optimal inverse filter in the $L^{4}$ sense. Differentiable I-PF provides the attacker's state estimate and learns the system parameters for unknown system dynamics in the inverse filtering context. The proposed I-PF yields accurate estimates, even when the forward filter assumptions are violated, and proves to be more reliable than the I-EKF.

\appendices
\vspace{-8pt}
\section{Derivation of (17) and (18)}\label{app:sampling}
The densities \eqref{eqn:ipf p density} and \eqref{eqn:ipf q density} follow from the conditional independence of observations and the estimation process assumed in the system model \eqref{eqn:state x}-\eqref{eqn:observation a} using the standard PF's SIS-based non-linear filtering technique \cite[Chapter~3.2]{ristic2003beyond}. Recall from Section~\ref{subsec:ipf formulation} that we consider the joint density $p(\hat{\mathbf{x}}_{0:k},\mathbf{y}_{1:k}|\mathbf{x}_{0:k},\mathbf{a}_{1:k})=p(\hat{\mathbf{x}}_{0:k},\mathbf{y}_{1:k}|\mathbf{a}_{k},\mathbf{x}_{0:k},\mathbf{a}_{1:k-1})$ to obtain I-PF's optimal sampling density and weights. Using Bayes' theorem, we have
\par\noindent\small
\begin{align}
    &p(\hat{\mathbf{x}}_{0:k},\mathbf{y}_{1:k}|\mathbf{x}_{0:k},\mathbf{a}_{1:k})=\nonumber\\
    &\frac{p(\mathbf{a}_{k}|\hat{\mathbf{x}}_{0:k},\mathbf{y}_{1:k},\mathbf{x}_{0:k},\mathbf{a}_{1:k-1})p(\hat{\mathbf{x}}_{0:k},\mathbf{y}_{1:k}|\mathbf{x}_{0:k},\mathbf{a}_{1:k-1})}{p(\mathbf{a}_{k}|\mathbf{x}_{0:k},\mathbf{a}_{1:k-1})}.\label{eqn:ipf joint density frac}
\end{align}
\normalsize
First, consider $p(\mathbf{a}_{k}|\hat{\mathbf{x}}_{0:k},\mathbf{y}_{1:k},\mathbf{x}_{0:k},\mathbf{a}_{1:k-1})=p(\mathbf{a}_{k}|\hat{\mathbf{x}}_{k},\hat{\mathbf{x}}_{0:k-1},\mathbf{y}_{1:k},\mathbf{x}_{0:k},\mathbf{a}_{1:k-1})$. From \eqref{eqn:observation a}, the defender's observation $\mathbf{a}_{k}$ depends only on the state estimate $\hat{\mathbf{x}}_{k}$ such that $p(\mathbf{a}_{k}|\hat{\mathbf{x}}_{0:k},\mathbf{y}_{1:k},\mathbf{x}_{0:k},\mathbf{a}_{1:k-1})=p(\mathbf{a}_{k}|\hat{\mathbf{x}}_{k})$.

Next, we simplify $p(\hat{\mathbf{x}}_{0:k},\mathbf{y}_{1:k}|\mathbf{x}_{0:k},\mathbf{a}_{1:k-1})$ in the numerator of \eqref{eqn:ipf joint density frac}. To this end, we have $p(\hat{\mathbf{x}}_{k},\hat{\mathbf{x}}_{0:k-1},\mathbf{y}_{1:k}|\mathbf{x}_{0:k},\mathbf{a}_{1:k-1})=p(\hat{\mathbf{x}}_{k}|\hat{\mathbf{x}}_{0:k-1},\mathbf{y}_{1:k},\mathbf{x}_{0:k},\mathbf{a}_{1:k-1})p(\hat{\mathbf{x}}_{0:k-1},\mathbf{y}_{1:k}|\mathbf{x}_{0:k},\mathbf{a}_{1:k-1})$. But, from \eqref{eqn:filter T}, given $T(\cdot)$, state estimate $\hat{\mathbf{x}}_{k}$ is a function of previous estimate $\hat{\mathbf{x}}_{k-1}$ and observation $\mathbf{y}_{k}$ such that $p(\hat{\mathbf{x}}_{k}|\hat{\mathbf{x}}_{0:k-1},\mathbf{y}_{1:k},\mathbf{x}_{0:k},\mathbf{a}_{1:k-1})$ simplifies to $p(\hat{\mathbf{x}}_{k}|\hat{\mathbf{x}}_{k-1},\mathbf{y}_{k})$. Also, $p(\mathbf{y}_{k},\hat{\mathbf{x}}_{0:k-1},\mathbf{y}_{1:k-1}|\mathbf{x}_{0:k},\mathbf{a}_{1:k-1})=p(\mathbf{y}_{k}|\hat{\mathbf{x}}_{0:k-1},\mathbf{y}_{1:k-1},\mathbf{x}_{0:k},\mathbf{a}_{1:k-1})p(\hat{\mathbf{x}}_{0:k-1},\mathbf{y}_{1:k-1}|\mathbf{x}_{0:k},\mathbf{a}_{1:k-1})$. From \eqref{eqn:observation y}, observation $\mathbf{y}_{k}$ depends only on the true state $\mathbf{x}_{k}$ such that $p(\mathbf{y}_{k}|\hat{\mathbf{x}}_{0:k-1},\mathbf{y}_{1:k-1},\mathbf{x}_{0:k},\mathbf{a}_{1:k-1})=p(\mathbf{y}_{k}|\mathbf{x}_{k})$. Lastly, the attacker's estimation process and the observations upto $(k-1)$-th time instant are independent of future state $\mathbf{x}_{k}$ given the past states, i.e., $p(\hat{\mathbf{x}}_{0:k-1},\mathbf{y}_{1:k-1}|\mathbf{x}_{0:k},\mathbf{a}_{1:k-1})=p(\hat{\mathbf{x}}_{0:k-1},\mathbf{y}_{1:k-1}|\mathbf{x}_{0:k-1},\mathbf{a}_{1:k-1})$. Now, substituting $p(\mathbf{a}_{k}|\hat{\mathbf{x}}_{0:k},\mathbf{y}_{1:k},\mathbf{x}_{0:k},\mathbf{a}_{1:k-1})=p(\mathbf{a}_{k}|\hat{\mathbf{x}}_{k})$ and $p(\hat{\mathbf{x}}_{0:k},\mathbf{y}_{1:k}|\mathbf{x}_{0:k},\mathbf{a}_{1:k-1})=p(\hat{\mathbf{x}}_{k}|\hat{\mathbf{x}}_{k-1},\mathbf{y}_{k})p(\mathbf{y}_{k}|\mathbf{x}_{k})p(\hat{\mathbf{x}}_{0:k-1},\mathbf{y}_{1:k-1}|\mathbf{x}_{0:k-1},\mathbf{a}_{1:k-1})$ in \eqref{eqn:ipf joint density frac} yields $p(\hat{\mathbf{x}}_{0:k},\mathbf{y}_{1:k}|\mathbf{x}_{0:k},\mathbf{a}_{1:k})\propto p(\mathbf{a}_{k}|\hat{\mathbf{x}}_{k})\\\times p(\hat{\mathbf{x}}_{k}|\hat{\mathbf{x}}_{k-1},\mathbf{y}_{k})p(\mathbf{y}_{k}|\mathbf{x}_{k})p(\hat{\mathbf{x}}_{0:k-1},\mathbf{y}_{1:k-1}|\mathbf{x}_{0:k-1},\mathbf{a}_{1:k-1})$ which is \eqref{eqn:ipf p density} provided in Section~\ref{subsec:ipf formulation}.

Now, we consider the sampling density \eqref{eqn:ipf q density} of our I-PF. The standard PF estimating states $\mathbf{x}_{0:k}$ given observations $\mathbf{y}_{1:k}$ uses a sampling density $\widetilde{q}(\mathbf{x}_{0:k}|\mathbf{y}_{1:k})$ that factorizes as $\widetilde{q}(\mathbf{x}_{0:k}|\mathbf{y}_{1:k})\doteq \widetilde{q}(\mathbf{x}_{k}|\mathbf{x}_{0:k-1},\mathbf{y}_{1:k})\widetilde{q}(\mathbf{x}_{0:k-1}|\mathbf{y}_{1:k-1})$ \cite[eq.~3.12]{ristic2003beyond}. In the case of I-PF, we are estimating the joint density of $(\hat{\mathbf{x}}_{0:k},\mathbf{y}_{1:k})$ given true states $\mathbf{x}_{0:k}$ and observations $\mathbf{a}_{1:k}$. Hence, we choose a sampling density $q(\cdot)$ such that $q(\hat{\mathbf{x}}_{0:k},\mathbf{y}_{1:k}|\mathbf{x}_{0:k},\mathbf{a}_{1:k})\doteq q(\hat{\mathbf{x}}_{k},\mathbf{y}_{k}|\hat{\mathbf{x}}_{0:k-1},\mathbf{y}_{1:k-1},\mathbf{x}_{0:k},\mathbf{a}_{1:k})\\q(\hat{\mathbf{x}}_{0:k-1},\mathbf{y}_{1:k-1}|\mathbf{x}_{0:k-1},\mathbf{a}_{1:k-1})$. Furthermore, when the standard PF requires only the filtered posterior $p(\mathbf{x}_{k}|\mathbf{y}_{1:k})$ at each time step, the sampling density $\widetilde{q}(\mathbf{x}_{k}|\mathbf{x}_{0:k-1},\mathbf{y}_{1:k})$ is commonly assumed to depend only on $\mathbf{x}_{k-1}$ and $\mathbf{y}_{k}$ \cite[Sec.~3.2]{ristic2003beyond}, i.e., $\widetilde{q}(\mathbf{x}_{k}|\mathbf{x}_{0:k-1},\mathbf{y}_{1:k})=\widetilde{q}(\mathbf{x}_{k}|\mathbf{x}_{k-1},\mathbf{y}_{k})$. Similarly, in our I-PF, we assume $q(\hat{\mathbf{x}}_{k},\mathbf{y}_{k}|\hat{\mathbf{x}}_{0:k-1},\mathbf{y}_{1:k-1},\mathbf{x}_{0:k},\mathbf{a}_{1:k})=q(\hat{\mathbf{x}}_{k},\mathbf{y}_{k}|\hat{\mathbf{x}}_{k-1},\mathbf{y}_{k-1},\mathbf{x}_{k},\mathbf{a}_{k})$. Overall, we choose I-PF's sampling density such that $q(\hat{\mathbf{x}}_{0:k},\mathbf{y}_{1:k}|\mathbf{x}_{0:k},\mathbf{a}_{1:k})=q(\hat{\mathbf{x}}_{k},\mathbf{y}_{k}|\hat{\mathbf{x}}_{k-1},\mathbf{y}_{k-1},\mathbf{x}_{k},\mathbf{a}_{k})q(\hat{\mathbf{x}}_{0:k-1},\mathbf{y}_{1:k-1}|\mathbf{x}_{0:k-1},\mathbf{a}_{1:k-1})$, which is \eqref{eqn:ipf q density} of Section~\ref{subsec:ipf formulation}.

\vspace{-8pt}
\section{Proof of Theorem~\ref{thm:ipf convergence}}
\label{app:ipf convergence}
We first restate Lemma~\ref{lemma:sum to power four inequality}-\ref{lemma:indicator inequality} from \cite{hu2008basic}. In Section~\ref{app:preliminaries}, some preliminary results are derived, from which the proof of Theorem~\ref{thm:ipf convergence} follows using mathematical induction.
\begin{lemma}
\label{lemma:sum to power four inequality}
Let $\{\xi_{i}\}_{1\leq i\leq n}$ be conditionally independent random variables given $\sigma$-algebra $\mathcal{G}$ such that $\mathbb{E}[\xi_{i}|\mathcal{G}]=0$ and $\mathbb{E}[|\xi_{i}|^{4}|\mathcal{G}]<\infty$. Then $\mathbb{E}\left[\left|\sum_{i=1}^{n}\xi_{i}\right|^{4}|\mathcal{G}\right]\leq\sum_{i=1}^{n}\mathbb{E}[|\xi_{i}|^{4}|\mathcal{G}]+\left(\sum_{i=1}^{n}\mathbb{E}[|\xi_{i}|^{2}|\mathcal{G}]\right)^{2}$.
\end{lemma}
\begin{lemma}
\label{lemma:diff to power p inequality}
    If $\mathbb{E}|\xi|^{p}<\infty$, then $\mathbb{E}|\xi-\mathbb{E}[\xi]|^{p}\leq 2^{p}\mathbb{E}|\xi|^{p}$, for any $p\geq 1$.
\end{lemma}
\begin{lemma}
\label{lemma:sum to max inequality}
    Let $\{\xi_{i}\}_{1\leq i\leq n}$ be conditionally independent random variables given $\sigma$-algebra $\mathcal{G}$ such that $\mathbb{E}[\xi_{i}|\mathcal{G}]=0$ and $\mathbb{E}[|\xi_{i}|^{4}|\mathcal{G}]<\infty$. Then $\mathbb{E}\left[\left|\frac{1}{n}\sum_{i=1}^{n}\xi_{i}\right|^{4}|\mathcal{G}\right]\leq 2\;\textrm{max}_{1\leq i\leq n}\mathbb{E}[|\xi_{i}|^{4}|\mathcal{G}]/n^{2}$.
\end{lemma}
\begin{lemma}
\label{lemma:indicator inequality}
    Denote $A^{c}$ as the complementary set of a given set $A$. Also, $I_{A}$ denotes the indicator function for a set $A$. Consider a random variable $\eta$ with probability density function $p(x)$ such that $\mathbb{P}(\eta\in A^{c})\leq\epsilon<1$. Define a random variable $\xi$ with probability density function as $\frac{p(x)I_{A}}{\mathbb{P}(A)}$ where $\mathbb{P}(A)=\int p(y)I_{A}dy$. If $\psi$ be a measurable function satisfying $\mathbb{E}[\psi^{2}(\eta)]<\infty$, then $|\mathbb{E}[\psi(\xi)]-\mathbb{E}[\psi(\eta)]|\leq\frac{2\sqrt{\mathbb{E}[\psi^{2}(\eta)]}}{1-\epsilon}\sqrt{\epsilon}$. In the case when $\mathbb{E}|\psi(\eta)|<\infty$, we have $\mathbb{E}|\psi(\xi)|\leq\frac{\mathbb{E}|\psi(\eta)|}{1-\epsilon}$.
\end{lemma}

\vspace{-8pt}
\subsection{Preliminaries to the proof}\label{app:preliminaries}
 Denote $\mathcal{F}_{k-1}=\sigma\{(\hat{\mathbf{x}}^{i}_{k-1},\mathbf{y}^{i}_{k-1})\;\textrm{for}\;1\leq i\leq N\}$ as the $\sigma$-algebra generated by particles at the previous $(k-1)$-th time step. For Lemma~\ref{lemma:infinite loop}-\ref{lemma:Sampling pow 4 terms}, we assume \eqref{eqn:ipf converge} holds for the previous time instant $(k-1)$. In \eqref{eqn:initialize diff} of Appendix~\ref{app:proof}, it is shown that \eqref{eqn:ipf converge} also holds for $k=0$. Further, we assume
 \par\noindent\small
 \begin{align}
   \mathbb{E}|\langle\pi_{k-1|k-1}^{N},|\phi|^{4}\rangle|\leq M_{k-1|k-1}\|\phi\|^{4}_{k-1,4}.\label{eqn:k-1 pow 4 one}  
 \end{align}
 \normalsize
where $M_{k-1|k-1}>0$ is a constant independent of the number of particles $N$. This inequality is necessary for the proof of the theorem and also proved to hold for all $k\geq 0$ in Section~\ref{app:proof}.
 
\begin{lemma}\label{lemma:infinite loop}
Consider the particles $\{(\hat{\overline{\mathbf{x}}}^{i}_{k},\overline{\mathbf{y}}^{i}_{k})\}$ drawn in the I-PF's importance sampling step. Assume that \eqref{eqn:ipf converge} holds for $(k-1)$-th time instant. Then, for sufficiently large $N$, we have $\mathbb{P}\left(\frac{1}{N}\sum_{i=1}^{N}\beta(\mathbf{a}_{k}|\hat{\overline{\mathbf{x}}}^{i}_{k})<\gamma_{k}|\mathcal{F}_{k-1}\right)<\epsilon_{k}$ for some $0<\epsilon_{k}<1$, which implies that the I-PF algorithm will not run into an infinite loop in steps $1$ and $2$.
\end{lemma}
\begin{proof}
Note that $(\hat{\overline{\mathbf{x}}}^{i}_{k},\overline{\mathbf{y}}^{i}_{k})$ are drawn from the distribution $\langle\pi^{N}_{k-1|k-1},\delta_{T}\rho\rangle$ in the importance sampling step such that
\par\noindent\small
\begin{align}
    \mathbb{E}[\phi(\hat{\overline{\mathbf{x}}}^{i}_{k},\overline{\mathbf{y}}^{i}_{k})|\mathcal{F}_{k-1}]=\langle\pi^{N}_{k-1|k-1},\delta_{T}\rho\phi\rangle,\label{eqn:expect of predicted particles}
\end{align}
\normalsize
Further, since $(\hat{\overline{\mathbf{x}}}^{i}_{k},\overline{\mathbf{y}}^{i}_{k})$ have equal weights (because of resampling at each step), we have $\frac{1}{N}\sum_{i=1}^{N}\beta(\mathbf{a}_{k}|\hat{\overline{\mathbf{x}}}^{i}_{k})=\langle\pi^{N}_{k-1|k-1},\delta_{T}\rho\beta\rangle$. Because of the modification, the distribution of $\{(\hat{\widetilde{\mathbf{x}}}^{i}_{k},\widetilde{\mathbf{y}}^{i}_{k})\}$ is obtained from the distribution of $\{(\hat{\overline{\mathbf{x}}}^{i}_{k},\overline{\mathbf{y}}^{i}_{k})\}$ conditioned on the event $\Gamma_{k}\doteq\{\langle\pi^{N}_{k-1|k-1},\delta_{T}\rho\beta\rangle\geq\gamma_{k}\}$. In order to use Lemma~\ref{lemma:indicator inequality} to handle the difference in the empirical distributions of $\{(\hat{\overline{\mathbf{x}}}^{i}_{k},\overline{\mathbf{y}}^{i}_{k})\}$ and $\{(\hat{\widetilde{\mathbf{x}}}^{i}_{k},\widetilde{\mathbf{y}}^{i}_{k})\}$, we require the probability of $\Gamma_{k}^{c}=\{\langle\pi^{N}_{k-1|k-1},\delta_{T}\rho\beta\rangle<\gamma_{k}\}$. Consider
\par\noindent\small
\begin{align*}
    &\mathbb{P}(\langle\pi^{N}_{k-1|k-1},\delta_{T}\rho\beta\rangle<\gamma_{k})=\\
    &\mathbb{P}(\langle\pi^{N}_{k-1|k-1},\delta_{T}\rho\beta\rangle-\langle\pi_{k-1|k-1},\delta_{T}\rho\beta\rangle<\gamma_{k}-\langle\pi_{k-1|k-1},\delta_{T}\rho\beta\rangle)\\
    &\leq \mathbb{P}(|\langle\pi^{N}_{k-1|k-1},\delta_{T}\rho\beta\rangle-\langle\pi_{k-1|k-1},\delta_{T}\rho\beta\rangle|>|\gamma_{k}-\langle\pi_{k-1|k-1},\delta_{T}\rho\beta\rangle|),
\end{align*}
\normalsize
because $\gamma_{k}-\langle\pi_{k-1|k-1},\delta_{T}\rho\beta\rangle<0$ from assumption \textbf{A1}. Using Markov's inequality and then, \eqref{eqn:ipf converge} replacing $k$ by $k-1$ with the indicator function being bounded by $1$ yields
\par\noindent\small
\begin{align*}
    &\mathbb{P}(\langle\pi^{N}_{k-1|k-1},\delta_{T}\rho\beta\rangle<\gamma_{k})\leq \frac{\mathbb{E}|\langle\pi^{N}_{k-1|k-1},\delta_{T}\rho\beta\rangle-\langle\pi_{k-1|k-1},\delta_{T}\rho\beta\rangle|^{4}}{|\gamma_{k}-\langle\pi_{k-1|k-1},\delta_{T}\rho\beta\rangle|^{4}}\\
    &\leq\frac{C_{k-1|k-1}\|\rho\|_{\infty}^{4}}{|\gamma_{k}-\langle\pi_{k-1|k-1},\delta_{T}\rho\beta\rangle|^{4}}\times\frac{\|\beta\|^{4}_{k-1,4}}{N^{2}}.
\end{align*}
\normalsize
Hence, the bound in the lemma will hold for some $0<\epsilon_{k}<1$ for sufficiently large $N$. In the following, we denote $C_{\gamma_{k}}=\frac{C_{k-1|k-1}\|\rho\|_{\infty}^{4}}{|\gamma_{k}-\langle\pi_{k-1|k-1},\delta_{T}\rho\beta\rangle|^{4}}$.
\end{proof}
\begin{lemma}\label{lemma:Pi123 bounds}
Consider the optimal filter's prediction distribution $\pi_{k|k-1}$ and its approximation $\widetilde{\pi}^{N}_{k|k-1}$ obtained in I-PF. Define
\par\noindent\small
\begin{align}
    &\Pi_{1}\doteq\langle\widetilde{\pi}^{N}_{k|k-1},\phi\rangle-\frac{1}{N}\sum_{i=1}^{N}\mathbb{E}[\phi(\hat{\widetilde{\mathbf{x}}}^{i}_{k},\widetilde{\mathbf{y}}^{i}_{k})|\mathcal{F}_{k-1}],\label{eqn:def Pi1}\\      
    &\Pi_{2}\doteq\frac{1}{N}\sum_{i=1}^{N}\mathbb{E}[\phi(\hat{\widetilde{\mathbf{x}}}^{i}_{k},\widetilde{\mathbf{y}}^{i}_{k})|\mathcal{F}_{k-1}]-\langle\pi^{N}_{k-1|k-1},\delta_{T}\rho\phi\rangle,\label{eqn:def Pi2}\\
    &\Pi_{3}\doteq\langle\pi^{N}_{k-1|k-1},\delta_{T}\rho\phi\rangle-\langle\pi_{k|k-1},\phi\rangle.\label{eqn:def Pi3}
\end{align}
\normalsize
Then, $\Pi_{1}$, $\Pi_{2}$ and $\Pi_{3}$ satisfy
\par\noindent\small
\begin{align}
    &\mathbb{E}[|\Pi_{1}|^{4}]\leq C_{\Pi_{1}}\frac{\|\phi\|^{4}_{k-1,4}}{N^{2}},\label{eqn:Pi1 term}\\
    &\mathbb{E}[|\Pi_{2}|^{4}]\leq C_{\Pi_{2}}\frac{\|\phi\|^{4}_{k-1,4}}{N^{2}},\label{eqn:Pi2 term}\\
    &\mathbb{E}[|\Pi_{3}|^{4}]\leq C_{\Pi_{3}}\frac{\|\phi\|^{4}_{k-1,4}}{N^{2}},\label{eqn:Pi3 term}
\end{align}
\normalsize
for suitable constants $C_{\Pi_{1}},C_{\Pi_{2}}$ and $C_{\Pi_{3}}>0$. Note that $\Pi_{1}$, $\Pi_{2}$ and $\Pi_{3}$ are defined for given $\pi_{k|k-1}$ and $\widetilde{\pi}^{N}_{k|k-1}$ while constants $C_{\Pi_{1}},C_{\Pi_{2}}$ and $C_{\Pi_{3}}$ do not dependent on $N$.
\end{lemma}
\begin{proof}
\textbf{(a) $\Pi_{1}$ term:} Recall that $\widetilde{\pi}^{N}_{k|k-1}$ is the empirical distribution obtained from particles $\{(\hat{\widetilde{\mathbf{x}}}^{i}_{k},\widetilde{\mathbf{y}}^{i}_{k})\}$. Hence,
\par\noindent\small
\begin{align*}
    &\mathbb{E}[|\Pi_{1}|^{4}|\mathcal{F}_{k-1}]\\
    &=\mathbb{E}\left[\left|\frac{1}{N}\sum_{i=1}^{N}(\phi(\hat{\widetilde{\mathbf{x}}}^{i}_{k},\widetilde{\mathbf{y}}^{i}_{k})-\mathbb{E}[\phi(\hat{\widetilde{\mathbf{x}}}^{i}_{k},\widetilde{\mathbf{y}}^{i}_{k})|\mathcal{F}_{k-1}])\right|^{4}|\mathcal{F}_{k-1}\right]\\
    &\leq\frac{1}{N^{4}}\sum_{i=1}^{N}\mathbb{E}[|\phi(\hat{\widetilde{\mathbf{x}}}^{i}_{k},\widetilde{\mathbf{y}}^{i}_{k})-\mathbb{E}[\phi(\hat{\widetilde{\mathbf{x}}}^{i}_{k},\widetilde{\mathbf{y}}^{i}_{k})|\mathcal{F}_{k-1}]|^{4}|\mathcal{F}_{k-1}]\\
    &+\frac{1}{N^{4}}\left(\sum_{i=1}^{N}\mathbb{E}[|\phi(\hat{\widetilde{\mathbf{x}}}^{i}_{k},\widetilde{\mathbf{y}}^{i}_{k})-\mathbb{E}[\phi(\hat{\widetilde{\mathbf{x}}}^{i}_{k},\widetilde{\mathbf{y}}^{i}_{k})|\mathcal{F}_{k-1}]|^{2}|\mathcal{F}_{k-1}]\right)^{2}\leq\frac{1}{N^{4}}\\
    &\hspace{-0.5cm}\times\left(\sum_{i=1}^{N}2^{4}\mathbb{E}[|\phi(\hat{\widetilde{\mathbf{x}}}^{i}_{k},\widetilde{\mathbf{y}}^{i}_{k})|^{4}|\mathcal{F}_{k-1}]+\left(\sum_{i=1}^{N}2^{2}\mathbb{E}[|\phi(\hat{\widetilde{\mathbf{x}}}^{i}_{k},\widetilde{\mathbf{y}}^{i}_{k})|^{2}|\mathcal{F}_{k-1}]\right)^{2}\right),
\end{align*}
\normalsize
where the first and second inequalities are obtained using Lemma~\ref{lemma:sum to power four inequality} and \ref{lemma:diff to power p inequality}, respectively. Since $(\hat{\overline{\mathbf{x}}}^{i}_{k},\overline{\mathbf{y}}^{i}_{k})$ are obtained from $(\hat{\widetilde{\mathbf{x}}}^{i}_{k},\widetilde{\mathbf{y}}^{i}_{k})$ such that $\Gamma_{k}$ (defined in Lemma~\ref{lemma:infinite loop}) occurs and $\mathbb{P}(\Gamma_{k}^{c})<\epsilon_{k}$, using Lemma~\ref{lemma:indicator inequality} yields
\par\noindent\small
\begin{align*}
    &\mathbb{E}[|\Pi_{1}|^{4}|\mathcal{F}_{k-1}]\leq\frac{2^{4}}{N^{4}}\\
    &\times\left(\sum_{i=1}^{N}\frac{\mathbb{E}[|\phi(\hat{\overline{\mathbf{x}}}^{i}_{k},\overline{\mathbf{y}}^{i}_{k})|^{4}|\mathcal{F}_{k-1}]}{1-\epsilon_{k}}+\left(\sum_{i=1}^{N}\frac{\mathbb{E}[|\phi(\hat{\overline{\mathbf{x}}}^{i}_{k},\overline{\mathbf{y}}^{i}_{k})|^{2}|\mathcal{F}_{k-1}]}{1-\epsilon_{k}}\right)^{2}\right)\\
    &=\frac{2^{4}}{N^{4}(1-\epsilon_{k})^{2}}\sum_{i=1}^{N}(1-\epsilon_{k})\mathbb{E}[|\phi(\hat{\overline{\mathbf{x}}}^{i}_{k},\overline{\mathbf{y}}^{i}_{k})|^{4}|\mathcal{F}_{k-1}]\\
    &+\frac{2^{4}}{N^{4}(1-\epsilon_{k})^{2}}\left(\sum_{i=1}^{N}\mathbb{E}[|\phi(\hat{\overline{\mathbf{x}}}^{i}_{k},\overline{\mathbf{y}}^{i}_{k})|^{2}|\mathcal{F}_{k-1}]\right)^{2}\leq\frac{2^{4}}{N^{4}(1-\epsilon_{k})^{2}}\\
    &\times\left(\sum_{i=1}^{N}\mathbb{E}[|\phi(\hat{\overline{\mathbf{x}}}^{i}_{k},\overline{\mathbf{y}}^{i}_{k})|^{4}|\mathcal{F}_{k-1}]+\left(\sum_{i=1}^{N}\mathbb{E}[|\phi(\hat{\overline{\mathbf{x}}}^{i}_{k},\overline{\mathbf{y}}^{i}_{k})|^{2}|\mathcal{F}_{k-1}]\right)^{2}\right),
\end{align*}
\normalsize
where the last inequality is obtained using $1-\epsilon_{k}<1$. Expressing the expectations as in \eqref{eqn:expect of predicted particles}, we have
\par\noindent\small
\begin{align*}
    &\mathbb{E}[|\Pi_{1}|^{4}|\mathcal{F}_{k-1}]\leq\frac{2^{4}}{N^{4}(1-\epsilon_{k})^{2}}\\
    &\times\left(\sum_{i=1}^{N}\langle\pi_{k-1|k-1}^{N},\delta_{T}\rho|\phi|^{4}\rangle+(\sum_{i=1}^{N}\langle\pi_{k-1|k-1}^{N},\delta_{T}\rho|\phi|^{2}\rangle)^{2}\right)\\
    &=\frac{2^{4}}{(1-\epsilon_{k})^{2}}\left(\frac{\langle\pi_{k-1|k-1}^{N},\delta_{T}\rho|\phi|^{4}\rangle}{N^{3}}+\frac{\langle\pi_{k-1|k-1}^{N},\delta_{T}\rho|\phi|^{2}\rangle^{2}}{N^{2}}\right)\\
    &\leq\frac{2^{4}}{(1-\epsilon_{k})^{2}}\left(\frac{\langle\pi_{k-1|k-1}^{N},\delta_{T}\rho|\phi|^{4}\rangle}{N^{3}}+\frac{\langle\pi_{k-1|k-1}^{N},\delta_{T}\rho|\phi|^{4}\rangle}{N^{2}}\right)\\
    &\leq\frac{2^{5}}{(1-\epsilon_{k})^{2}}\frac{\langle\pi_{k-1|k-1}^{N},\delta_{T}\rho|\phi|^{4}\rangle}{N^{2}},
\end{align*}
\normalsize
where the second last and last inequalities result from Jensen's inequality and $N>1$, respectively. Finally, using \eqref{eqn:k-1 pow 4 one} with the indicator function being bounded by $1$, we have $\mathbb{E}[|\Pi_{1}|^{4}]\leq\frac{2^{5}}{(1-\epsilon_{k})^{2}}\|\rho\|_{\infty} M_{k-1|k-1}\frac{\|\phi\|^{4}_{k-1,4}}{N^{2}}$ which gives \eqref{eqn:Pi1 term} with $C_{\Pi_{1}}\doteq\frac{2^{5}}{(1-\epsilon_{k})^{2}}\|\rho\|_{\infty} M_{k-1|k-1}$.

\textbf{(b) $\Pi_{2}$ term:} Using \eqref{eqn:expect of predicted particles}, we have
\par\noindent\small
\begin{align*}
    |\Pi_{2}|^{4}&=\left|\frac{1}{N}\sum_{i=1}^{N}(\mathbb{E}[\phi(\hat{\widetilde{\mathbf{x}}}^{i}_{k},\widetilde{\mathbf{y}}^{i}_{k})|\mathcal{F}_{k-1}]-\mathbb{E}[\phi(\hat{\overline{\mathbf{x}}}^{i}_{k},\overline{\mathbf{y}}^{i}_{k})|\mathcal{F}_{k-1}])\right|^{4}\\
    &\leq\frac{1}{N}\sum_{i=1}^{N}|\mathbb{E}[\phi(\hat{\widetilde{\mathbf{x}}}^{i}_{k},\widetilde{\mathbf{y}}^{i}_{k})|\mathcal{F}_{k-1}]-\mathbb{E}[\phi(\hat{\overline{\mathbf{x}}}^{i}_{k},\overline{\mathbf{y}}^{i}_{k})|\mathcal{F}_{k-1}]|^{4},
\end{align*}
\normalsize
from the Jensen's inequality. As mentioned earlier $\{(\hat{\overline{\mathbf{x}}}^{i}_{k},\overline{\mathbf{y}}^{i}_{k})\}$ are obtained from $\{(\hat{\widetilde{\mathbf{x}}}^{i}_{k},\widetilde{\mathbf{y}}^{i}_{k})\}$ such that $\Gamma_{k}$ occurs such that using Lemma~\ref{lemma:indicator inequality} and then Jensen's inequality, we obtain
\par\noindent\small
\begin{align*}
    |\Pi_{2}|^{4}&\leq\frac{1}{N}\sum_{i=1}^{N}\left(\frac{2\sqrt{\mathbb{E}[\phi(\hat{\overline{\mathbf{x}}}^{i}_{k},\overline{\mathbf{y}}^{i}_{k})^{2}|\mathcal{F}_{k-1}]}}{1-\epsilon_{k}}\sqrt{\epsilon_{k}}\right)^{4}\\
    &\leq\frac{2^{4}\epsilon_{k}^{2}}{N(1-\epsilon_{k})^{4}}\sum_{i=1}^{N}\mathbb{E}[\phi(\hat{\overline{\mathbf{x}}}^{i}_{k},\overline{\mathbf{y}}^{i}_{k})^{4}|\mathcal{F}_{k-1}].
\end{align*}
\normalsize
Substituting $\epsilon_{k}=C_{\gamma_{k}}\|\beta\|^{4}_{k-1,4}/N^{2}$ (from Lemma~\ref{lemma:infinite loop}) and \eqref{eqn:expect of predicted particles}, we have $|\Pi_{2}|^{4}\leq\frac{2^{4}}{(1-\epsilon_{k})^{4}}C_{\gamma_{k}}^{2}\frac{\|\beta\|^{8}_{k-1,4}}{N^{4}}\langle\pi^{N}_{k-1|k-1},\delta_{T}\rho\phi^{4}\rangle$. Denote $C'_{\Pi_{2}}=2^{4}C_{\gamma_{k}}^{2}\|\beta\|^{8}_{k-1,4}/(1-\epsilon_{k})^{4}$. Finally, using \eqref{eqn:k-1 pow 4 one} with the indicator function being bounded by $1$, we have $\mathbb{E}[|\Pi_{2}|^{4}]\leq C'_{\Pi_{2}}\frac{\mathbb{E}|\langle\pi^{N}_{k-1|k-1},\delta_{T}\rho\phi^{4}\rangle|}{N^{4}}\leq C'_{\Pi_{2}}M_{k-1|k-1}\|\rho\|_{\infty}\frac{\|\phi\|^{4}_{k-1,4}}{N^{4}}$ which yields \eqref{eqn:Pi2 term} using $N>1$ and $C_{\Pi_{2}}=C'_{\Pi_{2}}M_{k-1|k-1}\|\rho\|_{\infty}$.

\textbf{(c) $\Pi_{3}$ term:} Using \eqref{eqn:optimal filter time update} and \eqref{eqn:ipf converge} replacing $k$ by $k-1$,
\par\noindent\small
\begin{align*}
    &\mathbb{E}[|\Pi_{3}|^{4}]=\mathbb{E}|\langle\pi^{N}_{k-1|k-1},\delta_{T}\rho\phi\rangle-\langle\pi_{k|k-1},\phi\rangle|^{4}\\
    &=\mathbb{E}|\langle\pi^{N}_{k-1|k-1},\delta_{T}\rho\phi\rangle-\langle\pi_{k-1|k-1},\delta_{T}\rho\phi\rangle|^{4}\\
    &\leq C_{k-1|k-1}\|\rho\|_{\infty}^{4}\frac{\|\phi\|^{4}_{k-1,4}}{N^{2}},
\end{align*}
\normalsize
which yields \eqref{eqn:Pi3 term} defining $C_{\Pi_{3}}\doteq C_{k-1|k-1}\|\rho\|_{\infty}^{4}$.
\end{proof}
\begin{lemma}\label{lemma:Sampling pow 4 terms}
If \eqref{eqn:k-1 pow 4 one} holds, then
\par\noindent\small
\begin{align}
    &\mathbb{E}\left|\langle\widetilde{\pi}^{N}_{k|k-1},|\phi|^{4}\rangle-\frac{1}{N}\sum_{i=1}^{N}\mathbb{E}[|\phi(\hat{\widetilde{\mathbf{x}}}^{i}_{k},\widetilde{\mathbf{y}}^{i}_{k})|^{4}|\mathcal{F}_{k-1}]\right|\nonumber\\
    &\leq\frac{2}{(1-\epsilon_{k})}M_{k-1|k-1}\|\rho\|_{\infty}\|\phi\|^{4}_{k-1,4},\label{eqn:T1 predict pow 4}\\
    &\mathbb{E}\left|\frac{1}{N}\sum_{i=1}^{N}\mathbb{E}[|\phi(\hat{\widetilde{\mathbf{x}}}^{i}_{k},\widetilde{\mathbf{y}}^{i}_{k})|^{4}|\mathcal{F}_{k-1}]-\langle\pi^{N}_{k-1|k-1},\delta_{T}\rho|\phi|^{4}\rangle\right|\nonumber\\
    &\leq\frac{2-\epsilon_{k}}{1-\epsilon_{k}}M_{k-1|k-1}\|\rho\|_{\infty}\|\phi\|^{4}_{k-1,4}\label{eqn:T2 predict pow 4}
    \end{align}
    \begin{align}
    &\hspace{-0.5cm}\mathbb{E}|\langle\pi^{N}_{k-1|k-1},\delta_{T}\rho|\phi|^{4}\rangle-\langle\pi_{k|k-1},|\phi|^{4}\rangle|\nonumber\\
    &\;\;\;\leq \|\rho\|_{\infty}(M_{k-1|k-1}+1)\|\phi\|^{4}_{k-1,4}\label{eqn:T3 predict pow 4},
\end{align}
\normalsize
where $\epsilon_{k}$ is the same as defined in Lemma~\ref{lemma:infinite loop}.
\end{lemma}
\begin{proof}
Consider the first inequality \eqref{eqn:T1 predict pow 4}. Since $\widetilde{\pi}^{N}_{k|k-1}$ is the empirical distribution obtained from particles $\{(\hat{\widetilde{\mathbf{x}}}^{i}_{k},\widetilde{\mathbf{y}}^{i}_{k})\}$, we have
\par\noindent\small
\begin{align*}
    &\mathbb{E}\left[\mathbb{E}\left|\langle\widetilde{\pi}^{N}_{k|k-1},|\phi|^{4}\rangle-\frac{1}{N}\sum_{i=1}^{N}\mathbb{E}[|\phi(\hat{\widetilde{\mathbf{x}}}^{i}_{k},\widetilde{\mathbf{y}}^{i}_{k})|^{4}|\mathcal{F}_{k-1}]\right||\mathcal{F}_{k-1}\right]\\
    &=\frac{1}{N}\mathbb{E}\left[\mathbb{E}\left|\sum_{i=1}^{N}(|\phi(\hat{\widetilde{\mathbf{x}}}^{i}_{k},\widetilde{\mathbf{y}}^{i}_{k})|^{4}-\mathbb{E}[|\phi(\hat{\widetilde{\mathbf{x}}}^{i}_{k},\widetilde{\mathbf{y}}^{i}_{k})|^{4}|\mathcal{F}_{k-1}])\right||\mathcal{F}_{k-1}\right]\nonumber\\
    &\leq\frac{2}{N}\mathbb{E}[\sum_{i=1}^{N}\mathbb{E}[|\phi(\hat{\widetilde{\mathbf{x}}}^{i}_{k},\widetilde{\mathbf{y}}^{i}_{k})|^{4}|\mathcal{F}_{k-1}]]\leq\frac{2}{N(1-\epsilon_{k})}\\
    &\times\mathbb{E}[\sum_{i=1}^{N}\mathbb{E}[|\phi(\hat{\overline{\mathbf{x}}}^{i}_{k},\overline{\mathbf{y}}^{i}_{k})|^{4}|\mathcal{F}_{k-1}]]=\frac{2}{(1-\epsilon_{k})}\mathbb{E}[\langle\pi^{N}_{k-1|k-1},\delta_{T}\rho|\phi|^{4}\rangle],
\end{align*}
\normalsize
where the first and second inequalities are obtained using Lemma~\ref{lemma:diff to power p inequality} and \ref{lemma:indicator inequality}, respectively. The last equality follows because particles $\{(\hat{\overline{\mathbf{x}}}^{i}_{k},\overline{\mathbf{y}}^{i}_{k})\}$ are drawn from distribution $\langle\pi^{N}_{k-1|k-1},\delta_{T}\rho\rangle$. Finally, using \eqref{eqn:k-1 pow 4 one}, we obtain \eqref{eqn:T1 predict pow 4}.

Next, we consider the inequality \eqref{eqn:T2 predict pow 4}. Using \eqref{eqn:expect of predicted particles} replacing $\phi$ by $|\phi|^{4}$, we have
\par\noindent\small
\begin{align*}
    &\mathbb{E}\left|\frac{1}{N}\sum_{i=1}^{N}\mathbb{E}[|\phi(\hat{\widetilde{\mathbf{x}}}^{i}_{k},\widetilde{\mathbf{y}}^{i}_{k})|^{4}|\mathcal{F}_{k-1}]-\langle\pi^{N}_{k-1|k-1},\delta_{T}\rho|\phi|^{4}\rangle\right|\\
    &=\mathbb{E}\left|\frac{1}{N}\sum_{i=1}^{N}(\mathbb{E}[|\phi(\hat{\widetilde{\mathbf{x}}}^{i}_{k},\widetilde{\mathbf{y}}^{i}_{k})|^{4}|\mathcal{F}_{k-1}]-\mathbb{E}[|\phi(\hat{\overline{\mathbf{x}}}^{i}_{k},\overline{\mathbf{y}}^{i}_{k})|^{4}|\mathcal{F}_{k-1}])\right|\\
    &\leq\frac{1}{N}\sum_{i=1}^{N}\left(\mathbb{E}[\mathbb{E}[|\phi(\hat{\widetilde{\mathbf{x}}}^{i}_{k},\widetilde{\mathbf{y}}^{i}_{k})|^{4}|\mathcal{F}_{k-1}]]+\mathbb{E}[\mathbb{E}[|\phi(\hat{\overline{\mathbf{x}}}^{i}_{k},\overline{\mathbf{y}}^{i}_{k})|^{4}|\mathcal{F}_{k-1}]]\right)\nonumber.
\end{align*}
\normalsize
Now, using Lemma~\ref{lemma:indicator inequality}, we have
\par\noindent\small
\begin{align*}       
    &\mathbb{E}\left|\frac{1}{N}\sum_{i=1}^{N}\mathbb{E}[|\phi(\hat{\widetilde{\mathbf{x}}}^{i}_{k},\widetilde{\mathbf{y}}^{i}_{k})|^{4}|\mathcal{F}_{k-1}]-\langle\pi^{N}_{k-1|k-1},\delta_{T}\rho|\phi|^{4}\rangle\right|\\
    &\leq\frac{1}{N}\sum_{i=1}^{N}\left(\frac{\mathbb{E}[\mathbb{E}[|\phi(\hat{\overline{\mathbf{x}}}^{i}_{k},\overline{\mathbf{y}}^{i}_{k})|^{4}|\mathcal{F}_{k-1}]]}{1-\epsilon_{k}}+\mathbb{E}[\mathbb{E}[|\phi(\hat{\overline{\mathbf{x}}}^{i}_{k},\overline{\mathbf{y}}^{i}_{k})|^{4}|\mathcal{F}_{k-1}]]\right)\\
    &=\frac{2-\epsilon_{k}}{1-\epsilon_{k}}\mathbb{E}[\langle\pi^{N}_{k-1|k-1},\delta_{T}\rho|\phi|^{4}\rangle],\nonumber
\end{align*}
\normalsize
which yields \eqref{eqn:T2 predict pow 4} using assumption \eqref{eqn:k-1 pow 4 one}.

Finally, consider the last inequality \eqref{eqn:T3 predict pow 4}. From \eqref{eqn:optimal filter time update} and triangle inequality, we have
\par\noindent\small
\begin{align*}
    &\mathbb{E}|\langle\pi^{N}_{k-1|k-1},\delta_{T}\rho|\phi|^{4}\rangle-\langle\pi_{k|k-1},|\phi|^{4}\rangle|\leq\mathbb{E}|\langle\pi^{N}_{k-1|k-1},\delta_{T}\rho|\phi|^{4}\rangle|\\
    &\hspace{-0.5cm}+|\langle\pi_{k-1|k-1},\delta_{T}\rho|\phi|^{4}\rangle|\leq\|\rho\|_{\infty}M_{k-1|k-1}\|\phi\|^{4}_{k-1,4}+\|\rho\|_{\infty}\|\phi\|^{4}_{k-1,4},
\end{align*}
\normalsize
where the last inequality results from the assumption \eqref{eqn:k-1 pow 4 one} and the definition of $\|\phi\|_{k-1,4}$ (from Theorem~\ref{thm:ipf convergence}). Hence, \eqref{eqn:T3 predict pow 4} is proved.
\end{proof}

\vspace{-8pt}
\subsection{Proof of the theorem}\label{app:proof}
As mentioned earlier, the proof employs an induction framework. To this end, we prove that \eqref{eqn:ipf converge} holds for the $k$-th time step provided that the inequality holds for $(k-1)$-th time step. Additionally, we also show that
\par\noindent\small
\begin{align}
    \mathbb{E}|\langle\pi^{N}_{k|k},|\phi|^{4}\rangle|\leq M_{k|k}\|\phi\|^{4}_{k,4},\label{eqn:M inequality k}
\end{align}
\normalsize
which is \eqref{eqn:k-1 pow 4 one} with $k-1$ replaced by $k$. The first part of Theorem~\ref{thm:ipf convergence}, i.e., the algorithm does not run into an infinite loop, follows from Lemma~\ref{lemma:infinite loop}. For the second part of the theorem, we analyze the empirical distributions obtained in different steps of the I-PF algorithm and obtain the corresponding inequalities of \eqref{eqn:ipf converge} and \eqref{eqn:M inequality k}.

\noindent\textbf{Initialization:} Following the induction framework, we first show that \eqref{eqn:ipf converge} and \eqref{eqn:M inequality k} hold for $k=0$. Consider $\{\hat{\mathbf{x}}^{i}_{0}\}_{1\leq i\leq N}$ and $\{\mathbf{y}^{i}_{0}\}_{1\leq i\leq N}$ as the i.i.d. and mutually independent (initial) samples drawn from distributions $\widetilde{\pi}^{x}_{0}(d\hat{\mathbf{x}}_{0})$ and $\rho(\mathbf{y}_{0}|\mathbf{x}_{0})$, respectively. Recall that $\widetilde{\pi}^{x}_{0}$ is the initial distribution assumed in  I-PF. Denote $\pi_{0}$ as the joint distribution for particles $\{(\hat{\mathbf{x}}^{i}_{0},\mathbf{y}^{i}_{0})\}_{1\leq i\leq N}$ such that $\langle\pi_{0},\phi\rangle=\mathbb{E}[\phi(\hat{\mathbf{x}}^{i}_{0},\mathbf{y}^{i}_{0})]$ irrespective of $i$. Also, $\phi^{N}_{0}$ is the empirical distribution obtained from particles $\{(\hat{\mathbf{x}}^{i}_{0},\mathbf{y}^{i}_{0})\}_{1\leq i\leq N}$. We have $\mathbb{E}|\langle\pi_{0}^{N},\phi\rangle-\langle\pi_{0},\phi\rangle|^{4}=\mathbb{E}\left|\frac{1}{N}\sum_{i=1}^{N}\left(\phi(\hat{\mathbf{x}}^{i}_{0},\mathbf{y}^{i}_{0})-\mathbb{E}[\phi(\hat{\mathbf{x}}^{i}_{0},\mathbf{y}^{i}_{0})]\right)\right|^{4}$. Using Lemma~\ref{lemma:sum to max inequality}, we obtain $ \mathbb{E}|\langle\pi_{0}^{N},\phi\rangle-\langle\pi_{0},\phi\rangle|^{4}\leq\frac{2}{N^{2}}\mathbb{E}\left|\phi(\hat{\mathbf{x}}^{i}_{0},\mathbf{y}^{i}_{0})-\mathbb{E}[\phi(\hat{\mathbf{x}}^{i}_{0},\mathbf{y}^{i}_{0})]\right|^{4}$ because $(\hat{\mathbf{x}}^{i}_{0},\mathbf{y}^{i}_{0})$ are identically distributed for all $i=1,2,\hdots N$. Finally, using Lemma~\ref{lemma:diff to power p inequality} with $\mathbb{E}|\phi(\hat{\mathbf{x}}^{i}_{0},\mathbf{y}^{i}_{0})|^{4}=\langle\pi_{0},|\phi|^{4}\rangle$, we have
\par\noindent\small
\begin{align}
     \mathbb{E}|\langle\pi_{0}^{N},\phi\rangle-\langle\pi_{0},\phi\rangle|^{4}\leq\frac{32}{N^{2}}\|\phi\|^{4}_{0,4}\doteq C_{0|0}\frac{\|\phi\|^{4}_{0,4}}{N^{2}},\label{eqn:initialize diff}
\end{align}
\normalsize
because $\|\phi\|_{0,4}=\textrm{max}\{1,\langle\pi_{0},|\phi|^{4}\rangle\}^{1/4}$ from Theorem~\ref{thm:ipf convergence}.

Similarly, using Lemma~\ref{lemma:diff to power p inequality}, we obtain
\par\noindent\small
\begin{align}
&\mathbb{E}|\langle\pi_{0}^{N},|\phi|^{4}\rangle-\langle\pi_{0},|\phi|^{4}\rangle|\nonumber\\
&=\mathbb{E}\left|\frac{1}{N}\sum_{i=1}^{N}\left(|\phi(\hat{\mathbf{x}}^{i}_{0},\mathbf{y}^{i}_{0})|^{4}-\mathbb{E}|\phi(\hat{\mathbf{x}}^{i}_{0},\mathbf{y}^{i}_{0})|^{4}\right)\right|\nonumber\\
&\leq\frac{1}{N}\sum_{i=1}^{N}2\mathbb{E}|\phi(\hat{\mathbf{x}}^{i}_{0},\mathbf{y}^{i}_{0})|^{4}=2\mathbb{E}|\phi(\hat{\mathbf{x}}^{i}_{0},\mathbf{y}^{i}_{0})|^{4}.\label{eqn:initial M inter}
\end{align}
\normalsize
From triangle inequality, we have $\mathbb{E}|\langle\pi_{0}^{N},|\phi|^{4}\rangle|\leq\mathbb{E}|\langle\pi_{0}^{N},|\phi|^{4}\rangle-\langle\pi_{0},|\phi|^{4}\rangle|+|\langle\pi_{0},|\phi|^{4}\rangle|$. Note that $\pi^{N}_{0}$ is a random function obtained from the randomly generated particles, but $\pi_{0}$ is a given initial distribution. Hence, using \eqref{eqn:initial M inter} and $\langle\pi_{0},|\phi|^{4}\rangle=\mathbb{E}|\phi(\hat{\mathbf{x}}^{i}_{0},\mathbf{y}^{i}_{0})|^{4}$, we have
\par\noindent\small
\begin{align}
    \mathbb{E}|\langle\pi_{0}^{N},|\phi|^{4}\rangle|\leq 3\mathbb{E}|\phi(\hat{\mathbf{x}}^{i}_{0},\mathbf{y}^{i}_{0})|^{4}\leq\|\phi\|^{4}_{0,4}\doteq M_{0|0}\|\phi\|^{4}_{0,4}.\label{eqn:initialize pow 4}
\end{align}
\normalsize

Inequalities \eqref{eqn:initialize diff} and \eqref{eqn:initialize pow 4} show that \eqref{eqn:ipf converge} and \eqref{eqn:M inequality k} hold for $k=0$. Next, we assume that these inequalities hold at $(k-1)$-th time instant, i.e.,
\par\noindent\small
\begin{align}
    &\mathbb{E}|\langle\pi_{k-1|k-1}^{N},\phi\rangle-\langle\pi_{k-1|k-1},\phi\rangle|^{4}\leq C_{k-1|k-1}\frac{\|\phi\|^{4}_{k-1,4}}{N^{2}},\label{eqn:k-1 diff}\\
    &\mathbb{E}|\langle\pi_{k-1|k-1}^{N},|\phi|^{4}\rangle|\leq M_{k-1|k-1}\|\phi\|^{4}_{k-1,4},\label{eqn:k-1 pow 4}
\end{align}
\normalsize
where \eqref{eqn:k-1 pow 4} is same as \eqref{eqn:k-1 pow 4 one}, repeated here as a ready reference.

\noindent\textbf{Importance sampling with modification:} Recall that in the sampling step, we draw particles $\{(\hat{\overline{\mathbf{x}}}^{i}_{k},\overline{\mathbf{y}}^{i}_{k})\}$ until the inequality $\frac{1}{N}\sum_{i=1}^{N}\beta(\mathbf{a}_{k}|\hat{\overline{\mathbf{x}}}^{i}_{k})\geq\gamma_{k}$ is satisfied. Finally, $\{(\hat{\widetilde{\mathbf{x}}}^{i}_{k},\widetilde{\mathbf{y}}^{i}_{k})\}$ are the particles for which the condition holds, resulting in empirical distribution $\widetilde{\pi}^{N}_{k|k-1}$ as an approximation of $\pi_{k|k-1}$. Now, we consider bounds on $\mathbb{E}|\langle\widetilde{\pi}^{N}_{k|k-1},\phi\rangle-\langle\pi_{k|k-1},\phi\rangle|^{4}$ and $\mathbb{E}|\langle\widetilde{\pi}^{N}_{k|k-1},|\phi|^{4}\rangle-\langle\pi_{k|k-1},|\phi|^{4}\rangle|$.

Consider $\langle\widetilde{\pi}^{N}_{k|k-1},\phi\rangle-\langle\pi_{k|k-1},\phi\rangle=\Pi_{1}+\Pi_{2}+\Pi_{3}$ with $\Pi_{1}$, $\Pi_{2}$ and $\Pi_{3}$ as defined in \eqref{eqn:def Pi1}-\eqref{eqn:def Pi3} with appropriate bounds derived in Lemma~\ref{lemma:Pi123 bounds}. From Minkowski's inequality, we have $\mathbb{E}^{1/4}|\langle\widetilde{\pi}^{N}_{k|k-1},\phi\rangle-\langle\pi_{k|k-1},\phi\rangle|^{4}\leq\mathbb{E}^{1/4}|\Pi_{1}|^{4}+\mathbb{E}^{1/4}|\Pi_{2}|^{4}+\mathbb{E}^{1/4}|\Pi_{3}|^{4}$ such that using \eqref{eqn:Pi1 term}-\eqref{eqn:Pi3 term} yields
\par\noindent\small
\begin{align}
      \mathbb{E}|\langle\widetilde{\pi}^{N}_{k|k-1},\phi\rangle-\langle\pi_{k|k-1},\phi\rangle|^{4}\leq\widetilde{C}_{k|k-1}\frac{\|\phi\|^{4}_{k-1,4}}{N^{2}},\label{eqn:predict diff}
\end{align}
\normalsize
where constant $\widetilde{C}_{k|k-1}\doteq(C_{\Pi_{1}}^{1/4}+C_{\Pi_{2}}^{1/4}+C_{\Pi_{3}}^{1/4})^{4}$.

Next, we consider $\mathbb{E}|\langle\widetilde{\pi}^{N}_{k|k-1},|\phi|^{4}\rangle-\langle\pi_{k|k-1},|\phi|^{4}\rangle|$and employ a similar separation method. In particular, we have
\par\noindent\small
\begin{align*}
    &\langle\widetilde{\pi}^{N}_{k|k-1},|\phi|^{4}\rangle-\langle\pi_{k|k-1},|\phi|^{4}\rangle=\langle\widetilde{\pi}^{N}_{k|k-1},|\phi|^{4}\rangle\\
    &-\frac{1}{N}\sum_{i=1}^{N}\mathbb{E}[|\phi(\hat{\widetilde{\mathbf{x}}}^{i}_{k},\widetilde{\mathbf{y}}^{i}_{k})|^{4}|\mathcal{F}_{k-1}]+\frac{1}{N}\sum_{i=1}^{N}\mathbb{E}[|\phi(\hat{\widetilde{\mathbf{x}}}^{i}_{k},\widetilde{\mathbf{y}}^{i}_{k})|^{4}|\mathcal{F}_{k-1}]\\
    &-\langle\pi^{N}_{k-1|k-1},\delta_{T}\rho|\phi|^{4}\rangle+\langle\pi^{N}_{k-1|k-1},\delta_{T}\rho|\phi|^{4}\rangle-\langle\pi_{k|k-1},|\phi|^{4}\rangle.
\end{align*}
\normalsize
Using bounds \eqref{eqn:T1 predict pow 4}-\eqref{eqn:T3 predict pow 4} from Lemma~\ref{lemma:Sampling pow 4 terms}, we obtain
\par\noindent\small
\begin{align}
    \mathbb{E}|\langle\widetilde{\pi}^{N}_{k|k-1},|\phi|^{4}\rangle-\langle\pi_{k|k-1},|\phi|^{4}\rangle|\leq\widetilde{M}_{k|k-1}\|\phi\|^{4}_{k-1,4},\label{eqn:predict pow 4}
\end{align}
\normalsize
where constant $\widetilde{M}_{k|k-1}\doteq\|\rho\|_{\infty}\\\times\left(\frac{2}{1-\epsilon_{k}}M_{k-1|k-1}+\frac{2-\epsilon_{k}}{1-\epsilon_{k}}M_{k-1|k-1}+M_{k-1|k-1}+1\right)$. The inequalities \eqref{eqn:predict diff} and \eqref{eqn:predict pow 4} are the counterparts of inequalities \eqref{eqn:ipf converge} and \eqref{eqn:M inequality k}, respectively, for the (approximate) prediction distribution $\widetilde{\pi}^{N}_{k|k-1}$ obtained from the modified importance sampling in I-PF.

\noindent\textbf{Weight computation:} The posterior distribution $\widetilde{\pi}^{N}_{k|k}$ is obtained from the prediction distribution $\widetilde{\pi}^{N}_{k|k-1}$ by associating weight $\omega^{i}_{k}$ with each particle $(\hat{\widetilde{\mathbf{x}}}^{i}_{k},\mathbf{y}^{i}_{k})$. Hence, we now analyze $\mathbb{E}|\langle\widetilde{\pi}^{N}_{k|k},\phi\rangle-\langle\pi_{k|k},\phi\rangle|^{4}$ and $\mathbb{E}|\langle\widetilde{\pi}^{N}_{k|k},|\phi|^{4}\rangle|$ based on \eqref{eqn:predict diff} and \eqref{eqn:predict pow 4}. In particular, we obtain the counterparts of inequalities \eqref{eqn:ipf converge} and \eqref{eqn:M inequality k} for $\widetilde{\pi}^{N}_{k|k}$ in following Claims~\ref{claim:update claim diff} and \ref{claim:update claim pow}, respectively.
\begin{claim}\label{claim:update claim diff}
The optimal filter's posterior distribution $\pi_{k|k}$ and its approximation $\widetilde{\pi}^{N}_{k|k}$ in I-PF satisfy
\par\noindent\small
\begin{align}
    \mathbb{E}|\langle\widetilde{\pi}^{N}_{k|k},\phi\rangle-\langle\pi_{k|k},\phi\rangle|^{4}\leq\widetilde{C}_{k|k}\frac{\|\phi\|^{4}_{k-1,4}}{N^{2}},\label{eqn:update diff}
\end{align}
\normalsize
for suitable $\widetilde{C}_{k|k}>0$.
\end{claim}
\begin{claimproof}
Using \eqref{eqn:optimal filter measurement update}, we again employ a separation method as $\langle\widetilde{\pi}^{N}_{k|k},\phi\rangle-\langle\pi_{k|k},\phi\rangle=\frac{\langle\widetilde{\pi}^{N}_{k|k-1},\beta\phi\rangle}{\langle\widetilde{\pi}^{N}_{k|k-1},\beta\rangle}-\frac{\langle\pi_{k|k-1},\beta\phi\rangle}{\langle\pi_{k|k-1},\beta\rangle}\doteq\widetilde{\Pi}_{1}+\widetilde{\Pi}_{2}$ where $\widetilde{\Pi}_{1}=\frac{\langle\widetilde{\pi}^{N}_{k|k-1},\beta\phi\rangle}{\langle\widetilde{\pi}^{N}_{k|k-1},\beta\rangle}-\frac{\langle\widetilde{\pi}^{N}_{k|k-1},\beta\phi\rangle}{\langle\pi_{k|k-1},\beta\rangle}$ and $\widetilde{\Pi}_{2}=\frac{\langle\widetilde{\pi}^{N}_{k|k-1},\beta\phi\rangle}{\langle\pi_{k|k-1},\beta\rangle}-\frac{\langle\pi_{k|k-1},\beta\phi\rangle}{\langle\pi_{k|k-1},\beta\rangle}$. As noted in Remark~\ref{remark:threshold}, the modification step implies that $\langle\widetilde{\pi}^{N}_{k|k},\beta\rangle\geq\gamma_{k}$. Hence, under assumptions \textbf{A2} and \textbf{A3}, we have
\par\noindent\small
\begin{align*}
    &|\widetilde{\Pi}_{1}|=\left|\frac{\langle\widetilde{\pi}^{N}_{k|k-1},\beta\phi\rangle}{\langle\widetilde{\pi}^{N}_{k|k-1},\beta\rangle}\times\frac{\langle\pi_{k|k-1},\beta\rangle-\langle\widetilde{\pi}^{N}_{k|k-1},\beta\rangle}{\langle\pi_{k|k-1},\beta\rangle}\right|\\
    &\leq\frac{\|\beta\phi\|_{\infty}}{\gamma_{k}\langle\pi_{k|k-1},\beta\rangle}|\langle\pi_{k|k-1},\beta\rangle-\langle\widetilde{\pi}^{N}_{k|k-1},\beta\rangle|.
\end{align*}
\normalsize
Now, using Minkowski's inequality, we have
\par\noindent\small
\begin{align*}
    &\mathbb{E}^{1/4}|\langle\widetilde{\pi}^{N}_{k|k},\phi\rangle-\langle\pi_{k|k},\phi\rangle|^{4}\leq\mathbb{E}^{1/4}|\widetilde{\Pi}_{1}|^{4}+\mathbb{E}^{1/4}|\widetilde{\Pi}_{2}|^{4}\\
    &\leq\frac{\|\beta\phi\|_{\infty}}{\gamma_{k}\langle\pi_{k|k-1},\beta\rangle}\mathbb{E}^{1/4}|\langle\pi_{k|k-1},\beta\rangle-\langle\widetilde{\pi}^{N}_{k|k-1},\beta\rangle|^{4}\\
    &+\frac{1}{\langle\pi_{k|k-1},\beta\rangle}\mathbb{E}^{1/4}|\langle\widetilde{\pi}^{N}_{k|k-1},\beta\phi\rangle-\langle\pi_{k|k-1},\beta\phi\rangle|^{4}.
\end{align*}
\normalsize
Finally, using \eqref{eqn:predict diff}, we obtain
\par\noindent\small
\begin{align*}
    &\mathbb{E}^{1/4}|\langle\widetilde{\pi}^{N}_{k|k},\phi\rangle-\langle\pi_{k|k},\phi\rangle|^{4}\leq\frac{\|\beta\phi\|_{\infty}}{\gamma_{k}\langle\pi_{k|k-1},\beta\rangle}\widetilde{C}^{1/4}_{k|k-1}\frac{\|\beta\|_{\infty}}{N^{1/2}}+\frac{1}{\langle\pi_{k|k-1},\beta\rangle}\\
    &\times\widetilde{C}^{1/4}_{k|k-1}\|\beta\|_{\infty}\frac{\|\phi\|_{k-1,4}}{N^{1/2}}\leq\frac{\widetilde{C}^{1/4}_{k|k-1}\|\beta\|_{\infty}}{\gamma_{k}\langle\pi_{k|k-1},\beta\rangle}(\|\beta\phi\|_{\infty}+\gamma_{k})\frac{\|\phi\|_{k-1,4}}{N^{1/2}},
\end{align*}
\normalsize
where the last inequality follows because by definition $\|\phi\|_{k-1,4}>1$. Defining $\widetilde{C}^{1/4}_{k|k}\doteq\frac{\widetilde{C}^{1/4}_{k|k-1}\|\beta\|_{\infty}}{\gamma_{k}\langle\pi_{k|k-1},\beta\rangle}(\|\beta\phi\|_{\infty}+\gamma_{k})$ proves the claim.
\end{claimproof}
\begin{claim}\label{claim:update claim pow}
The distributions $\pi_{k|k}$ and $\widetilde{\pi}^{N}_{k|k}$ satisfy
\par\noindent\small
\begin{align}
   \mathbb{E}|\langle\widetilde{\pi}^{N}_{k|k},|\phi|^{4}\rangle|\leq\widetilde{M}_{k|k}\|\phi\|^{4}_{k,4}\label{eqn:update pow 4},
\end{align}
\normalsize
for suitable $\widetilde{M}_{k|k}>0$.
\end{claim}
\begin{claimproof}
Using a similar separation method as in proof of Claim~\ref{claim:update claim diff}, we have
\par\noindent\small
\begin{align*}
   &\mathbb{E}|\langle\widetilde{\pi}^{N}_{k|k},|\phi|^{4}\rangle-\langle\pi_{k|k},|\phi|^{4}\rangle|\leq\mathbb{E}\left|\frac{\langle\widetilde{\pi}^{N}_{k|k-1},\beta|\phi|^{4}\rangle-\langle\pi_{k|k-1},\beta|\phi|^{4}\rangle}{\langle\pi_{k|k-1},\beta\rangle}\right|\\
   &+\mathbb{E}\left|\frac{\langle\widetilde{\pi}^{N}_{k|k-1},\beta|\phi|^{4}\rangle}{\langle\widetilde{\pi}^{N}_{k|k-1},\beta\rangle}\times\frac{\langle\pi_{k|k-1},\beta\rangle-\langle\widetilde{\pi}^{N}_{k|k-1},\beta\rangle}{\langle\pi_{k|k-1},\beta\rangle}\right|.
\end{align*}
\normalsize
Again using assumption \textbf{A2} and \eqref{eqn:predict pow 4}, we obtain
\par\noindent\small
\begin{align*}
   &\mathbb{E}|\langle\widetilde{\pi}^{N}_{k|k},|\phi|^{4}\rangle-\langle\pi_{k|k},|\phi|^{4}\rangle|\leq\frac{\|\beta\phi^{4}\|_{\infty}}{\gamma_{k}\langle\pi_{k|k-1},\beta\rangle}\\
   &\times\mathbb{E}|\langle\pi_{k|k-1},\beta\rangle-\langle\widetilde{\pi}^{N}_{k|k-1},\beta\rangle|+\frac{1}{\langle\pi_{k|k-1},\beta\rangle}\widetilde{M}_{k|k-1}\|\beta\|_{\infty}\|\phi\|^{4}_{k-1,4}\\
   &\leq\frac{\|\beta\phi^{4}\|_{\infty}2\|\beta\|_{\infty}}{\gamma_{k}\langle\pi_{k|k-1},\beta\rangle}\|\phi\|^{4}_{k-1,4}+\frac{\widetilde{M}_{k|k-1}\|\beta\|_{\infty}}{\langle\pi_{k|k-1},\beta\rangle}\|\phi\|^{4}_{k-1,4},
\end{align*}
\normalsize
because by definition $\|\phi\|_{k-1,4}>1$. Hence, $\mathbb{E}|(\widetilde{\pi}^{N}_{k|k},|\phi|^{4})|\leq\frac{\|\beta\phi^{4}\|_{\infty}2\|\beta\|_{\infty}}{\gamma_{k}\langle\pi_{k|k-1},\beta\rangle}\|\phi\|^{4}_{k-1,4}+\frac{\widetilde{M}_{k|k-1}\|\beta\|_{\infty}}{\langle\pi_{k|k-1},\beta\rangle}\|\phi\|^{4}_{k-1,4}+\langle\pi_{k|k},|\phi|^{4}\rangle$. But, $\langle\pi_{k|k},|\phi|^{4}\rangle\leq\|\phi\|^{4}_{k,4}$ and $\|\phi\|_{k,4}$ is increasing in $k$ such that $\mathbb{E}|(\widetilde{\pi}^{N}_{k|k},|\phi|^{4})|\leq 3\;\textrm{max}\left\{\frac{\|\beta\phi^{4}\|_{\infty}2\|\beta\|_{\infty}}{\gamma_{k}\langle\pi_{k|k-1},\beta\rangle},\frac{\widetilde{M}_{k|k-1}\|\beta\|_{\infty}}{\langle\pi_{k|k-1},\beta\rangle},1\right\}\|\phi\|^{4}_{k,4}$ which proves the claim with $\widetilde{M}_{k|k}\doteq3\;\textrm{max}\left\{\frac{\|\beta\phi^{4}\|_{\infty}2\|\beta\|_{\infty}}{\gamma_{k}\langle\pi_{k|k-1},\beta\rangle},\frac{\widetilde{M}_{k|k-1}\|\beta\|_{\infty}}{\langle\pi_{k|k-1},\beta\rangle},1\right\}$.
\end{claimproof}

\noindent\textbf{Resampling:} In this step, we draw $N$ independent particles $(\hat{\mathbf{x}}^{i}_{k},\mathbf{y}^{i}_{k})$ from the posterior distribution $\widetilde{\pi}^{N}_{k|k}$ and obtain the empirical distribution $\pi^{N}_{k|k}$ with equally weighted particles. Note that $\pi^{N}_{k|k}$ also approximates $\pi_{k|k}$ and is the I-PF's output posterior distribution. Now, we finally show that \eqref{eqn:ipf converge} and \eqref{eqn:M inequality k} hold true if we assume \eqref{eqn:k-1 diff} and \eqref{eqn:k-1 pow 4} hold at $(k-1)$-th time, which completes the induction proof. To this end, we analyze $\mathbb{E}|\langle\pi^{N}_{k|k},\phi\rangle-\langle\pi_{k|k},\phi\rangle|^{4}$ and $\mathbb{E}|\langle\pi^{N}_{k|k},|\phi|^{4}\rangle|$ based on \eqref{eqn:update diff} and \eqref{eqn:update pow 4}, and prove the following claims.
\begin{claim}\label{claim:k diff}
The distribution $\pi_{k|k}$ and its approximation $\pi^{N}_{k|k}$ satisfy $\mathbb{E}|\langle\pi^{N}_{k|k},\phi\rangle-\langle\pi_{k|k},\phi\rangle|^{4}\leq C_{k|k}\frac{\|\phi\|^{4}_{k,4}}{N^{2}}$.
\end{claim}
\begin{claimproof}
Consider the separation $\langle\pi^{N}_{k|k},\phi\rangle-\langle\pi_{k|k},\phi\rangle=\overline{\Pi}_{1}+\overline{\Pi}_{2}$ where $\overline{\Pi}_{1}=\langle\pi^{N}_{k|k},\phi\rangle-\langle\widetilde{\pi}^{N}_{k|k},\phi\rangle$ and $\overline{\Pi}_{2}=\langle\widetilde{\pi}^{N}_{k|k},\phi\rangle-\langle\pi_{k|k},\phi\rangle$. Denote $\mathcal{G}_{k}$ as the $\sigma$-algebra generated by $\{(\hat{\widetilde{\mathbf{x}}}^{i}_{k},\widetilde{\mathbf{y}}^{i}_{k})\}_{i=1}^{N}$. Since $(\hat{\mathbf{x}}^{i}_{k},\mathbf{y}^{i}_{k})$ are drawn independently from $\widetilde{\pi}^{N}_{k|k}$, we have $\mathbb{E}[\phi(\hat{\mathbf{x}}^{i}_{k},\mathbf{y}^{i}_{k})|\mathcal{G}_{k}]=\langle\widetilde{\pi}^{N}_{k|k},\phi\rangle$ such that $\overline{\Pi}_{1}=\frac{1}{N}\sum_{i=1}^{N}(\phi(\hat{\mathbf{x}}^{i}_{k},\mathbf{y}^{i}_{k})-\mathbb{E}[\phi(\hat{\mathbf{x}}^{i}_{k},\mathbf{y}^{i}_{k})|\mathcal{G}_{k}])$. Using Lemma~\ref{lemma:sum to max inequality} and \ref{lemma:diff to power p inequality}, and finally \eqref{eqn:update pow 4}, we obtain
\par\noindent\small
\begin{align}
\mathbb{E}[|\overline{\Pi}_{1}|^{4}|\mathcal{G}_{k}]\leq 2^{5}\widetilde{M}_{k|k}\frac{\|\phi\|^{4}_{k,4}}{N^{2}}.\label{eqn:resample P1}
\end{align}
\normalsize
Again, using the Minkowski's inequality, and \eqref{eqn:update diff} and \eqref{eqn:resample P1}, we have
\par\noindent\small
\begin{align*}
    &\mathbb{E}^{1/4}|\langle\pi^{N}_{k|k},\phi\rangle-\langle\pi_{k|k},\phi\rangle|^{4}\leq\mathbb{E}^{1/4}|\overline{\Pi}_{1}|^{4}+\mathbb{E}^{1/4}|\overline{\Pi}_{2}|^{4}\leq (2^{5}\widetilde{M}_{k|k})^{1/4}\\
    &\times\frac{\|\phi\|_{k,4}}{N^{1/2}}+\widetilde{C}^{1/4}_{k|k}\frac{\|\phi\|_{k-1,4}}{N^{1/2}}\leq ((2^{5}\widetilde{M}_{k|k})^{1/4}+\widetilde{C}^{1/4}_{k|k})\frac{\|\phi\|_{k,4}}{N^{1/2}},
\end{align*}
\normalsize
because $\|\phi\|_{k,4}$ is increasing in $k$. Hence, defining $C_{k|k}^{1/4}\doteq ((2^{5}\widetilde{M}_{k|k})^{1/4}+\widetilde{C}^{1/4}_{k|k})$ yields the inequality in the claim.
\end{claimproof}
\begin{claim}\label{claim:k pow 4}
The distributions $\pi_{k|k}$ and $\pi^{N}_{k|k}$ satisfy $\mathbb{E}|\langle\pi^{N}_{k|k},|\phi|^{4}\rangle|\leq M_{k|k}\|\phi\|^{4}_{k,4}$.
\end{claim}
\begin{claimproof}
Since, $(\hat{\mathbf{x}}^{i}_{k},\mathbf{y}^{i}_{k})\sim\widetilde{\pi}^{N}_{k,k}$, we have $\langle\widetilde{\pi}^{N}_{k|k},|\phi|^{4}\rangle=\mathbb{E}[|\phi(\hat{\mathbf{x}}^{i}_{k},\mathbf{y}^{i}_{k})|^{4}|\mathcal{G}_{k}]$. Then, using Lemma~\ref{lemma:diff to power p inequality} and \eqref{eqn:update pow 4}, we obtain
\par\noindent\small
\begin{align*}
    &\mathbb{E}|\langle\pi^{N}_{k|k},|\phi|^{4}\rangle-\langle\pi_{k|k},|\phi|^{4}\rangle|\\
    &\leq\mathbb{E}|\langle\pi^{N}_{k|k},|\phi|^{4}\rangle-\langle\widetilde{\pi}^{N}_{k|k},|\phi|^{4}\rangle|+\mathbb{E}|\langle\widetilde{\pi}^{N}_{k|k},|\phi|^{4}\rangle-\langle\pi_{k|k},|\phi|^{4}\rangle|\\
    &\leq\mathbb{E}\left|\frac{1}{N}\sum_{i=1}^{N}(|\phi(\hat{\mathbf{x}}^{i}_{k},\mathbf{y}^{i}_{k})|^{4}-\mathbb{E}[|\phi(\hat{\mathbf{x}}^{i}_{k},\mathbf{y}^{i}_{k})|^{4}|\mathcal{G}_{k}])\right|+\mathbb{E}|\langle\widetilde{\pi}_{k|k}^{N},|\phi|^{4}\rangle|\\
    &\hspace{-0.5cm}+\langle\pi_{k|k},|\phi|^{4}\rangle\leq\frac{1}{N}\sum_{i=1}^{N}2\mathbb{E}[\mathbb{E}[|\phi(\hat{\mathbf{x}}^{i}_{k},\mathbf{y}^{i}_{k})|^{4}|\mathcal{G}_{k}]]+\widetilde{M}_{k|k}\|\phi\|^{4}_{k,4}+\|\phi\|^{4}_{k,4}\\
    &=2\mathbb{E}[(\widetilde{\pi}^{N}_{k|k},|\phi|^{4})]+(\widetilde{M}_{k|k}+1)\|\phi\|^{4}_{k,4}\leq(3\widetilde{M}_{k|k}+1)\|\phi\|^{4}_{k,4},
\end{align*}
\normalsize
Hence, $\mathbb{E}|\langle\pi^{N}_{k|k},|\phi|^{4}\rangle|\leq(3\widetilde{M}_{k|k}+2)\|\phi\|^{4}_{k,4}$ which proves the claim with $M_{k|k}\leq(3\widetilde{M}_{k|k}+2)$.
\end{claimproof}

Claims~\ref{claim:k diff} and \ref{claim:k pow 4} show that inequalities \eqref{eqn:ipf converge} and \eqref{eqn:M inequality k} hold for $k$-th time instant, completing the induction proof.

\balance
\bibliographystyle{IEEEtran}
\bibliography{references}
\end{document}